\documentclass{article}

\usepackage{amsmath,amssymb,amsthm}
\usepackage[english]{babel}
\usepackage{mathrsfs}
\usepackage[T1]{fontenc}

\newtheorem{theorem}{Theorem}
\newtheorem{corollary}[theorem]{Corollary}
\newtheorem{proposition}[theorem]{Proposition}
\newtheorem{lemma}[theorem]{Lemma}
\newtheorem*{theorem*}{Theorem}

\theoremstyle{definition}
\newtheorem{definition}{Definition}

\theoremstyle{remark}
\newtheorem{remark}{Remark}
\newtheorem{remarks}[remark]{Remarks}
\newtheorem{example}{Example}
\newtheorem{examples}[example]{Examples}

\newcommand{\HH}{\mathbb{H}}
\newcommand{\OO}{\Omega}
\newcommand{\C}{\mathbb{C}}

\newcommand{\R}{\mathbb{R}}

\newcommand{\sd}{\partial_s }

\newcommand{\is}{\mathbb{S}}

\newcommand\Oc{\mathbb O}
\def\SS{\mathbb S}
\newcommand\Ac{A_{\C}}

\newcommand\I{\mathcal I}
\newcommand\B{\mathcal B}
\newcommand\Sl{\mathcal S}

\def\dd#1#2{\dfrac{\partial#1}{\partial#2}} 
 
\def\dt#1#2{{\partial#1}/{\partial#2}}

\newcommand\IM{\operatorname{Im}}
\newcommand\RE{\operatorname{Re}}
\newcommand\Pin{\operatorname{Pin}}
\newcommand\Spin{\operatorname{Spin}}

\newcommand\Reg{\mathcal {R}}
\newcommand\SR{\Sl\Reg(\OO_D)}

\begin{document}

\title{Slice regular functions on real alternative algebras}
\author{R.\ Ghiloni, A.\ Perotti
\thanks{Work partially supported by MIUR (PRIN Project ``Propriet\`a geometriche delle variet\`a reali e com\-ples\-se") and GNSAGA of INdAM}
\\Department of Mathematics\\
  University of Trento\\ Via Sommarive, 14\\ I--38123 Povo Trento ITALY\\
  perotti@science.unitn.it}

\maketitle

\begin{abstract} 
In this paper we develop a theory of slice regular functions on a real alternative algebra $A$. Our approach is based on a well--known Fueter's construction.
Two recent function theories can be included in our general theory: the one of slice regular functions of a quaternionic or octonionic variable 
and the theory of slice monogenic functions of a Clifford variable.
Our approach permits to extend the range of these function theories and to obtain new results. In particular, 
we get a strong form of the fundamental theorem of algebra for an ample class of polynomials with coefficients in $A$ 
and we prove a Cauchy integral formula for slice functions of class $C^1$.

\end{abstract}

Keywords: {Functions of a hypercomplex variable, Quaternions, Octonions, Fundamental theorem of algebra, Cauchy integral formula}

Math.\ Subj.\ Class: {30C15, 30G35, 32A30,  17D05}

\section{Introduction}
\label{sec:Introduction}

In \cite{Fueter1934}, R.\ Fueter proposed a simple method, which is now known as Fueter's Theorem, to generate quaternionic regular functions using complex holomorphic functions (we refer the reader to \cite{Su} and \cite{GHS} for the theory of Fueter regular functions). 
Given a holomorphic function (the ``stem function'')
\[
F(z)=u(\alpha,\beta)+ i\, v(\alpha,\beta)\quad 
\text{($z=\alpha+i\beta$ complex, $u,v$  real--valued)}
\]
in the upper complex half-plane, real--valued on $\R$, the formula
\[f(q):=u\left(q_0,|\IM(q)|\right)+\frac{\IM(q)}{|\IM(q)|}v\left(q_0,|\IM(q)|\right)\]
($q=q_0+q_1i+q_2j+q_3k\in\HH$, $\IM(q)=q_1i+q_2j+q_3k$) gives rise to a \emph{radially holomorphic} function on $\HH$, whose Laplacian $\Delta f$ is {Fueter regular}.
Fueter's construction was later extended to higher dimensions by Sce \cite{Sce},  Qian \cite{Qian1997} and 
Sommen \cite{Sommen2000} in the setting of octonionic and Clifford analysis. 

By means of a slight modification of Fueter's construction, it is possible to obtain a more general class of functions. It is the class of \emph{slice regular} (or \emph{Cullen regu\-lar}) functions of a quaternionic variable that was recently introduced by Gentili and Struppa \cite{GeSt2006CR,GeSt2007Adv}.  This notion of regularity do not coincide with the one of Fueter. In fact, the set of slice regular functions on a ball $B_R$ centered in the origin of $\HH$ coincides with that of all power series $\sum_iq^ia_i$ that converges in $B_R$, which fail to be Fueter regular.
If $u$ and $v$ are \emph{$\HH$--valued} functions on the upper complex half-plane, then the same formula given above
defines a {Cullen regular} function on $\HH$, whose Laplacian $\Delta f$ is still {Fueter regular}.

In the present paper, we extend Fueter's construction in order to develop a theory of slice regular functions on a real alternative algebra $A$. 
These functions will be obtained simply by taking $A$--valued components $u,v$ of the stem function $F$. As stated before, if $A$ is the algebra of quaternions, we get the theory of Cullen regular functions. If $A$ is the algebra of octonions, we obtain the corresponding theory of regular functions already considered in \cite{GeStRocky,GeStoStVl} and \cite{GhPePreprint2009}.
If $A$ is the Clifford algebra $\R_n$ with signature $(0,n)$ (cf.\ e.g.\ \cite{GHS} for properties of these algebras), slice functions will be defined on a proper subset of $\R_n$, what we call the \emph{quadratic cone} of the algebra. In particular, by restricting the Clifford variables to the subspace of paravectors, which is contained in the quadratic cone, we get the theory of \emph{slice monogenic} functions introduced by Colombo, Sabadini and Struppa in \cite{CoSaSt2009Israel}.

We describe in more detail the structure of the paper.
In Section~\ref{sec:TheQuadraticCone}, we introduce some basic definitions about real alternative algebras with an antiinvolution. We define the normal cone and the quadratic cone of an algebra $A$ and prove (Proposition~\ref{pro3}) that the quadratic cone is a union of complex planes of $A$. This property is the starting point for the extension of Fueter's construction. In Section~\ref{sec:SliceFunctions}, we introduce complex intrinsic functions with values in the complexified algebra $A\otimes_\R\C$ and use them as stem functions to generate $A$--valued  \emph{(left) slice functions}.
This approach, not being based upon power series, does not require the holomorphy of the stem function. Moreover, slice functions can be defined also on domains which do not intersect the \emph{real axis} of $A$ (i.e.\ the subspace generated by the unity of $A$). In Propositions~\ref{pro4} and~\ref{pro5}, we show that the natural domains of definition for slice functions are ``circular'' domains of $A$, those that are invariant for the action of the square roots of $-1$. 

In Section~\ref{sec:SliceRegularFunctions}, we restrict our attention to slice functions with \emph{holomorphic}  stem function, what we call \emph{(left) slice regular} functions on $A$. These functions forms on every circular domain $\OO_D$ a real vector space $\SR$ 
that is not closed w.r.t.\  the pointwise product in $A$. However, the pointwise product for stem functions in the complexified algebra  induces a natural product on slice functions (cf.\ Section~\ref{sec:Productofslicefunctions}), that generalizes the usual product of polynomials and power series. 

In Sections~\ref{sec:NormalFunctionAndAdmissibility} and~\ref{sec:ZerosOfSliceFunctions}, we study the zero set of slice and slice regular functions. To this aim, a fundamental tool is the  \emph{normal function} of a slice function, which is defined by means of the product and the antiinvolution. This concept leads us to restrict our attention to \emph{admissible} slice regular functions, which preserve many relevant properties of classical holomorphic functions.
We generalize to our setting a structure theorem for the zero set proved   by Pogorui and Shapiro \cite{PogoruiShapiro} for quaternionic polynomials and by  Gentili and Stoppato \cite{GeSto2008Mich} for quaternionic power series. A Remainder Theorem (Theorem~\ref{Remainder_Theorem}) gives us the possibility to define a notion of multiplicity for zeros of an admissible slice regular function.

In Section~\ref{sec:FTA}, we prove a version of the fundamental theorem of algebra  for slice regular admissible {polynomials}.
This theorem was proved for quaternionic polynomials by Eilenberg and Niven \cite{Ni1941,EiNi1944} and for octonionic polynomials by Jou \cite{Jou}.
In \cite[pp.\ 308ff]{EiSt1952}, Eilenberg and Steenrod gave a topological proof of the theorem valid for a class of real algebras including $\C$, $\HH$ and $\Oc$. See also \cite{Topuridze}, \cite{RoO} and \cite{GeStVl2008Zeit} for other proofs.
Gordon and Motzkin \cite{GordonMotzkin} proved, for polynomials on a (associative) division ring, that the number of conjugacy classes containing zeros of $p$ cannot be greater than the degree $m$ of $p$. This estimate was improved on the quaternions by Pogorui and Shapiro \cite{PogoruiShapiro}: if $p$ has $s$ spherical zeros and $l$ non--spherical zeros, then $2s+l\le m$. 
Gentili and Struppa \cite{GeSt2008Milan} showed that, using the right definition of multiplicity, the number of zeros of $p$ equals the degree of the polynomial.
In \cite{GhPePreprint2009}, this strong form was generalized to the octonions. Recently, Colombo, Sabadini and Struppa \cite{CoSaSt2009Israel,CoSaStPreprintIsrael} and Yang and Qian \cite{YangQianActa} proved some results on the structure of the set of zeros of a   
polynomial with paravector coefficients in a Clifford algebra.

We obtain a strong form of the fundamental theorem of algebra, which contains and generalizes the above results. We prove (Theorem~\ref{FTA}) that the sum of the multiplicities of the zeros of a slice regular admissible polynomial is equal to its degree.

Section~\ref{sec:CauchyIntegralFormulaForSliceFunctions} contains a Cauchy integral formula for slice functions of class $C^1$. We define a \emph{Cauchy kernel}, which in the  quaternionic case was already introduced in  \cite{CoGeSaPreprint2008}, and for slice monogenic functions in \cite{CoSaPreprint2008,CoSaTM2009}.  This kernel was applied in \cite{CoSaStPreprint2009Mich} to get Cauchy formulas for $C^1$ functions on a class of domains intersecting the real axis.

\section[The quadratic cone]{The quadratic cone of a real alternative algebra}
\label{sec:TheQuadraticCone}

Let $A$ be a finite--dimensional  real alternative algebra with a unity. We refer to \cite{Numbers,Schafer} for the main properties of such algebras. We assume that $A$ has dimension $d>1$ as a real vector space. We will identify the field of real numbers with the subalgebra of $A$ generated by the unity. Recall that an algebra is alternative if the \emph{associator} $(x,y,z):=(xy)z-x(yz)$ of three elements of $A$ is an alternating function of its arguments.

In a real algebra, we can consider the \emph{imaginary space} consisting of all non--real elements whose square is real (cf.\ \cite[\S8.1]{Numbers}). 
\begin{definition}
Let 
$\IM(A):=\{x\in A\ |\ x^2\in\R,\ x\notin\R\setminus\{0\}\}.$
 The elements of\, $\IM(A)$ are called \emph{purely imaginary} elements of $A$.
\end{definition}

\begin{remark}
As we will see in the examples below,  in general, the imaginary space $\IM(A)$ is not a vector subspace of $A$. However, if $A$ is a quadratic algebra, then a Frobenius's Lemma tells that $\IM(A)$ is a  subspace of $A$ (cf.\ e.g.\ \cite[\S8.2.1]{Numbers}). 
\end{remark}

In what follows, we will assume that on $A$ an involutory antiautomorphism (also called an \emph{antiinvolution}) is fixed. It is a linear map $x\mapsto x^c$ of $A$ into $A$ satisfying the following properties:

\begin{itemize}
\item
		$(x^c)^c=x\quad \forall x\in A$
\item
		$(xy)^c=y^cx^c\quad \forall x,y\in A$
\item
		$x^c=x\text{\quad for every real }x$.
	\end{itemize}



\begin{definition}
For every element $x$ of $A$, the \emph{trace} of $x$ is $t(x):=x+x^c\in A$ and the (squared) \emph{norm} of $x$ is
$n(x):=xx^c\in A$. Note that we \emph{do not} assume that $t(x)$ and $n(x)$ are real for every $x\in A$.
\end{definition}


Now we introduce the main objects of this section, two subsets of $A$ that will be called respectively the \emph{normal cone} and the \emph{quadratic cone} of the algebra. The second name is justified by two properties proved below: the quadratic cone is a real cone whose elements satisfy a real quadratic equation.

\begin{definition}\label{def3}
We call \emph{normal cone} of the algebra $A$  the subset
\[N_A:=\{0\}\cup\{x\in A\ |\ n(x)=n(x^c)\text{ is a real nonzero number\,} \}.\]
The \emph{quadratic cone} of the algebra $A$ is the set
	\[Q_A:=\R\cup\{x\in A\ |\ t(x)\in\R,\ n(x)\in\R,\ 4n(x)> t(x)^2\}.\]
We also set
$\SS_A:=\{J\in Q_A\ |\ J^2=-1\}\subseteq \IM(A)$. Elements of $\SS_A$ will be called \emph{square roots of $-1$} in the algebra $A$.
\end{definition}


\begin{examples}
\label{ex1}

$(1)$
Let $A$ be the associative non--commutative algebra  $\HH$ of the quaternions or the alternative non--associative algebra $\Oc$ of the octonions. Let $x^c=\bar x$ be the usual conjugation mapping. Then $\IM(A)$ is a subspace of $A$, $A=\R\oplus\IM(A)$ and $Q_\HH=\HH$, $Q_\Oc=\Oc$. In these cases, $\SS_\HH$ is a two-dimensional sphere and $\SS_\Oc$ is a six-dimensional sphere. 

$(2)$
Let $A$ be the real Clifford algebra $Cl_{0,n}=\R_n$ with the \emph{Clifford conjugation} defined by
 \begin{align*}
 x^c&=([x]_0+[x]_1+[x]_2+[x]_3+[x]_4+\cdots)^c\\
 &=[x]_0-[x]_1-[x]_2+[x]_3+[x]_4-\cdots,
 \end{align*}
where $[x]_k$ denotes the $k$--vector component of $x$ in $\R_n$ (cf.\ for example \cite[\S 3.2]{GHS} for definitions and properties of Clifford numbers). If $n\ge3$, $\IM(A)$ is not  a subspace of $A$.
For example, the basis elements $e_1$, $e_{23}$ belong to $\IM(A)$, but their sum $e_1+e_{23}$ is not purely imaginary. 

For any $n>1$, the subspace of \emph{paravectors}, 
\[\R^{n+1}:=\{x\in\R_n\ |\ [x]_k=0\text{ for every }k>1\}\]
is a (proper) subset of the quadratic cone $Q_{\R_n}$. For $x\in\R^{n+1}$, $t(x)=2x_0\in\R$ and $n(x)=|x|^2\ge0$ (the euclidean norm). The $(n-1)$--dimensional sphere 
$\SS=\{x=x_1e_1+\cdots+ x_ne_n\in\R^{n+1}\ |\ x_1^2+\cdots +x_n^2=1\}$
of unit 1--vectors is (properly) contained in $\SS_{\R_n}$. The normal cone contains also the \emph{Clifford group} $\Gamma_n$ and its subgroups $\Pin(n)$ and $\Spin(n)$.

$(3)$
We consider in more detail the case of $\R_3$.
An element $x\in\R_3$ can be represented as a sum 
\[x=x_0+\sum_{i=1}^3x_ie_i+\mathop{\sum_{j,k=1}^{3}}_{j<k}x_{jk}e_{jk}+x_{123}e_{123} 
\]
with real coefficients $x_0,x_i, x_{jk}, x_{123}$.
A computation similar to the one made in the proof of Proposition 1 in \cite{GeSt2008COV} shows that 
\[\IM({\R_3})=\{x\in\R_3\ |\ x_0=0,\ x_1x_{23}-x_2x_{13}+x_3x_{12}=0\},\]
that the normal cone of $\R_3$ is
\[N_{\R_3}=\{x\in\R_3\ |\ x_0x_{123}-x_1x_{23}+x_2x_{13}-x_3x_{12}=0\}\]
and the quadratic cone is the six--dimensional real algebraic set
\[Q_{\R_3}=\{x\in\R_3\ |\ x_{123}=0,\ x_1x_{23}-x_2x_{13}+x_3x_{12}=0\}.\]
Finally, $\SS_{\R_3}$ is the intersection of a $5$--sphere with the hypersurface of $\R_3$ with equation $x_1x_{23}-x_2x_{13}+x_3x_{12}=0$:
\[\SS_{\R_3}=\{x\in Q_{\R_3}\ |\ x_0=0,\textstyle\sum_ix_i^2+\sum_{j,k}x_{jk}^2=1\}\subset \IM({\R_3})\cap Q_{\R_3}.\]

\end{examples}

\begin{remark}
As seen in the examples, the quadratic cone $Q_A$ needs not be a subalgebra or a subspace of $A$.
\end{remark}

\begin{proposition}\label{Properties_of_Q_A}
Let $A$ be a  real alternative algebra with a fixed antiinvolution $x\mapsto x^c$. The following statements hold.
\begin{itemize}
	\item[$(1)$]
	$\R\subseteq Q_A$ and $\alpha x\in Q_A\ \forall \alpha\in\R,x\in Q_A$.
		\item[$(2)$]
	For every $x\in Q_A, \alpha\in\R$, the sum $\alpha+x$ belongs to $Q_A$.
	\item[$(3)$]
	Every $x\in Q_A$ satisfies a real {quadratic equation}
	\[x^2-x\,t(x)+n(x)=0.\]
		\item[$(4)$]
		$x\in Q_A$ is equivalent to $x^c\in Q_A$. Moreover, $Q_A\subseteq N_A$.
	\item[$(5)$]
	Every nonzero $x\in N_A$ is invertible: $x^{-1}=\dfrac{x^c}{n(x)}.$
	\item[$(6)$]
$\SS_A=\{J\in A\;|\; t(J)=0,\; n(J)=1\}$.
	\item[$(7)$]
	$Q_A=A$ if and only if  $A$ is isomorphic to one of the division algebras $\C,\HH$ or $\Oc$.
\end{itemize}
\end{proposition}
\begin{proof}
The first two statements are immediate consequences of the linearity of the antiinvolution.  For every $x\in Q_A$, the equality $n(x)=n(x^c)$ holds, since $x^c=t(x)-x$ commutes with $x$. From this fact comes the quadratic equation and also the fourth and fifth statement. We recall Artin's Theorem for alternative algebras (cf.\ \cite{Schafer}): the subalgebra generated by two elements is always associative. If $J\in\SS_A$, then $1=n(J^2)=(JJ)(J^c J^c)=J\,n(J^c)J^c=n(J)^2$. Then $n(J)=1$ and $J(J+J^c)=-1+n(J)=0$, from which it follows, multiplying by $J$ on the left, that $J+J^c=0$. Conversely, if $J+J^c=0$ and $JJ^c=1$, then $J^2=-JJ^c=-1$.

The last property follows from  properties (4),(5) and a result proved by Frobenius and Zorn (cf.\ \cite[\S8.2.4]{Numbers} and \S9.3.2), which states that if $A$ is a {quadratic,  real alternative algebra without divisors of zero}, then $A$ is isomorphic to one of the algebras $\C,\HH$ or $\Oc$.
\end{proof}

 
\begin{proposition}\label{pro2}
Let $A$ be a real alternative algebra with a fixed antiinvolution $x\mapsto x^c$. 
For every $x\in Q_A$, there exist uniquely determined elements $x_0\in\R$, $y\in \IM(A)\cap Q_A$, with $t(y)=0$, such that 
\[x=x_0+y.\]
\end{proposition}
\begin{proof}
Since $x\in Q_A$, $x_0:=(x+x^c)/2=t(x)/2$ is real. It remains to prove that $y:=(x-x^c)/2$ belongs to $\IM(A)$ when $x\in Q_A\setminus\R$.
Since $y+y^c=0$, $y\notin\R\setminus\{0\}$.
Condition $4n(x)> t(x)^2$ given in the definition of the quadratic cone assures that $y$ belongs to $\IM(A)$:
\[4y^2=-4n(y)=(x-x^c)^2=x^2-2n(x)+(x^c)^2=\]\[=t(x)x-n(x)-2n(x)+t(x^c)x^c-n(x^c)=t(x)^2-4n(x)<0.\]
Uniqueness follows from the equality $t(x)=t(x_0)=x_0$.
\end{proof}

Using the notation of the above proposition, for every $x\in Q_A$, we set $\RE(x):=x_0=\frac{x+x^c}2$, $\IM(x):=y=\frac{x-x^c}2$. Therefore $x=\RE(x)+\IM(x)$ for every $x\in Q_A$. The norm of $x\in Q_A$ is given by the formula
\[n(x)=(x_0+y)(x_0+y^c)=x_0^2-y^2=(\RE(x))^2-(\IM(x))^2.\]

For every $J\in \SS_A$ square root of $-1$, we will denote by $\C_J:=\langle 1,J\rangle\simeq\C$  the subalgebra of $A$ generated by  $J$. 

\begin{proposition}\label{pro3}
Under the same assumptions of the preceding proposition, if $Q_A\ne\R$, the following statements hold:
\begin{itemize}
\item[$(1)$]
$Q_A=\bigcup_{J\in \SS_A}\C_J$.
 \item[$(2)$]
 If  $I,J\in\SS_A$, $I\ne\pm J$, then $\C_I\cap\C_J=\R$.
\end{itemize}
\end{proposition}
\begin{proof}
If $x\in Q_A\setminus\R$ and $y=\IM(x)\ne0$, set $J:={y}/{\sqrt{n(y)}}$. Then $J^2=y^2/{n(y)}=-1$ and therefore $J\in\SS_A$ and $x=\RE(x)+\sqrt{n(y)}J\in\C_J$. Conversely, from Proposition~\ref{Properties_of_Q_A} (properties (1),(2)) every complex plane $\C_J$ is contained in $Q_A$. 

The second statement follows from the Independence Lemma for alternative algebras (cf.\ \cite[\S8.1]{Numbers}): since two elements $I,J\in\SS_A$, $I\ne\pm J$, are linearly independent, then also the triple $\{1,I,J\}$ is linearly independent. This last condition is equivalent to $\C_I\cap\C_J=\R$.
\end{proof}

\begin{corollary}
For every $x\in Q_A$, also the powers $x^k$ belong to $Q_A$.
\end{corollary} 

\section[Slice functions]{Slice functions}
\label{sec:SliceFunctions}

\subsection{A--stem functions}
\label{sec:AstemFunctions}
In this and the following sections, $A$ will denote a real alternative algebra with a fixed antiinvolution $x\mapsto x^c$. We will always assume that $Q_A\ne\R$, i.e.\ that $\SS_A\ne\emptyset$.

Let $\Ac=A\otimes_{\R}\C$ be the complexification of $A$. We will use the representation
  \[\Ac=\{w=x+iy\ |\ x,y\in A\}\quad  (i^2=-1).\]  
  $\Ac$ is a complex alternative algebra with a unity w.r.t.\ the product given by the formula
  \[(x+iy)(x'+iy')=xx'-yy'+i(xy'+yx').\]
  The algebra $A$ can be identified with the real subalgebra $A':=\{w=x+iy\ |\ y=0\}$ of $\Ac$ and the unity of $\Ac$ then coincides with the one of $A$.
In  $\Ac$ two commuting operators are defined: the \emph{complex--linear antiinvolution} $w\mapsto w^c=(x+iy)^c=x^c+iy^c$ and the
  \emph{complex conjugation} defined by $\overline w=\overline{x+iy}=x-iy$.

\begin{definition}
Let $D\subseteq\C$ be an open subset. If a function $F: D\rightarrow\Ac$ is \emph{complex intrinsic},  i.e.\ it satisfies the condition \[F(\overline z)=\overline{F(z)}\text{\quad for every $z\in D$ such that $\overline z\in D$},\] then $F$  is called an \emph{$A$--stem function} on $D$.
\end{definition}

The notion of \emph{stem function} was introduced by Fueter in \cite{Fueter1934} for complex--valued functions in order to construct \emph{radially holomorphic} functions on the space of quaternions.

In the preceding definition, there is no restriction to assume that $D$ is symmetric w.r.t.\ the real axis, i.e.\ $D=\text{conj}(D):=\{z\in\C\ |\ \bar z\in D\}$. In fact, if this is not the case, the function $F$ can be extended to $D\cup\text{conj}(D)$ by imposing complex intrinsicity. It must be noted, however, that the extended set will be non--connected if $D\cap\R=\emptyset$.

\begin{remarks}\label{rem1}
$(1)$ 
A function $F$ is an $A$--stem function if and only if	the $A$--valued components $F_1,F_2$ of $F=F_1+iF_2$ form an \emph{even--odd pair} w.r.t.\ the imaginary part of $z$, i.e.
	\[F_1(\overline z)=F_1(z),\  F_2(\overline z)=-F_2(z)\text{\quad for every $z\in D$.}\]

$(2)$
	Let $d$ be the dimension of $A$ as a real vector space. By means of a basis $\B=\{u_k\}_{k=1,\ldots, d}$ of $A$, $F$ can be identified with a complex intrinsic curve in $\C^d$. Let $F(z)=F_1(z)+iF_2(z)=\sum_{k=1}^d F_\B^k(z)u_k$, with $F_\B^k(z)\in\C$. Then \[\widetilde{F}_\B=(F_\B^1,\ldots, F_\B^d):D\rightarrow\C^d\]
	satisfies $\widetilde{F}_\B(\overline z)=\overline{\widetilde{F}_\B(z)}$. Giving to $A$ the unique manifold structure as a real vector space, we get that a stem function $F$ is of class $C^k$ ($k=0,1,\ldots,\infty$) or real--analytic if and only if the same property holds for $\widetilde{F}_\B$. This notion is clearly independent of the choice of the basis of $A$.
\end{remarks}

\subsection[Slice functions]{A--valued slice functions}
\label{sec:Avaluedslicefunctions}

Given an open subset $D$ of $\C$, let $\OO_D$ be the subset of $A$ obtained by the action on $D$ of the square roots of $-1$:
\[\OO_D:=\{x=\alpha+\beta J\in\C_J\ |\ \alpha,\beta\in\R,\ \alpha+i\beta\in D,\ J\in\SS_A\}.\]
Sets of this type will be called \emph{circular sets} in $A$. It follows from Proposition~\ref{pro3} that $\OO_D$ is a relatively open subset of the quadratic cone $Q_A$. 

\begin{definition}
Any stem function $F:D\rightarrow\Ac$ induces a \emph{left slice function} $f=\I(F):\OO_D\rightarrow A$. If $x=\alpha+\beta J\in D_J:=\OO_D\cap \C_J$, we set  
\[f(x):=F_1(z)+JF_2(z)\quad (z=\alpha+i\beta).\]
\end{definition}

The slice function $f$ is well--defined, since $(F_1,F_2)$ is an even--odd pair w.r.t.\ $\beta$ and then
$f(\alpha+(-\beta)(-J))=F_1(\overline z)+(-J)F_2(\overline z)=F_1(z)+JF_2(z)$. There is an analogous definition for \emph{right} slice functions when the element $J\in\SS_A$ is placed on the right of $F_2(z)$. In what follows, 
the term \emph{slice functions} will always mean left slice functions.
 
We will denote the set of (left) slice functions on $\OO_D$ by
\[\Sl(\OO_D):=\{f:\OO_D\rightarrow A\ |\ f=\I(F),\ F:D\rightarrow A_{\C} \text{ $A$--stem function}\}.\]

\begin{remark}
$\Sl(\OO_D)$ is a real vector space, since $\I(F+G)=\I(F)+\I(G)$ and $\I(Fa)=\I(F)a$ for every complex intrinsic functions $F$, $G$ and every $a\in \R$. 
\end{remark}

\begin{examples}
\label{ex2}
Assume $A=\HH$ or $A=\Oc$, with the usual conjugation mapping. 

$(1)$
For any element $a\in A$,
 $F(z):=z^na=\RE(z^n)a+i\, (\IM(z^n)a)$ induces the monomial $f(x)=x^na\in\Sl(A)$.

$(2)$
 By linearity,
 we get all the \emph{standard polynomials}\, $p(x)=\sum_{j=0}^n x^ja_j$ with right quaternionic or octonionic coefficients. More generally, every convergent \emph{power series}\, $\sum_jx^ja_j$, with (possibly infinite) convergence radius $R$ (w.r.t.\ $|x|=\sqrt{n(x)}$), belongs to the space $\Sl(B_R)$, where $B_R$ is the open ball of $A$ centered in the origin with radius $R$.

$(3)$
 Also $G(z):=\RE(z^n)a$ and  $H(z):=i\, (\IM(z^n)a)$ are complex intrinsic on $\C$. They induce respectively the slice functions
  $g(x)=\RE(x^n)a$ and $h(x)=f(x)-g(x)=(x^n-\RE(x^n))a$ on $A$.
 The difference $G(z)-H(z)=\bar z^n a$ induces $g(x)-h(x)=(2\RE(x^n)-x^n)a=
 \bar x^na\in\Sl(A)$.

The above examples generalize to standard polynomials in $x=\I(z)$ and $x^c=\I(\bar z)$ with coefficients in  $A$. The domain of slice polynomial functions or series must be restricted to subsets of the quadratic cone. For example, when $A=\R_3$, standard polynomials in $x$ with right Clifford coefficients can be considered as slice functions for $x\in\R_3$ such that $x_{123}=0,\ x_1x_{23}-x_2x_{13}+x_3x_{12}=0$ (cf.\ Example~\ref{ex1}(3) in Section \ref{sec:TheQuadraticCone}). In particular, they are defined on the space of {paravectors} $\R^4=\{x\in\R_3\ |\ x_{12}=x_{13}=x_{23}=x_{123}=0\}$. 
\end{examples}

\subsection{Representation formulas}
\label{sec:RepresentationFormulas}

For an element $J\in\SS_A$, let $\C_J^+$ denote the upper half plane
\[\C_J^+=\{x=\alpha+\beta J\in A\ |\ \beta\ge0\}.\]

\begin{proposition}\label{pro4}
Let $J,K\in\SS_A$ with $J-K$ invertible.
Every slice function $f\in\Sl(\OO_D)$ is uniquely determined by its values on the {two} distinct half planes $\C_J^+$ and $\C_K^+$. In particular, taking $K=-J$, we have that $f$ is uniquely determined by its values on a complex plane $\C_J$.
\end{proposition}
\begin{proof}
For $f\in\Sl(\OO_D)$, let $f_J^+$ be the restriction 
\[f_J^+:=f\,_{|\C_J^+\cap\OO_D}\ :\ \C_J^+\cap\OO_D\rightarrow A.\]
It is sufficient to show that the stem function $F$ such that $\I(F)=f$ can be recovered from the restrictions $f_J^+$, $f_K^+$. This follows from the formula
\[f_J^+(\alpha+\beta J)-f_K^+(\alpha+\beta K)=(J-K)F_2(\alpha+i\beta)\]
which holds for every $\alpha+i\beta\in D$ with $\beta\ge0$. In particular, it implies the vanishing of $F_2$ when $\beta=0$. Then $F_2$ is determined also for $\beta<0$ by imposing oddness w.r.t.\ $\beta$. Moreover, the formula
\[f_J^+(\alpha+\beta J)-JF_2(\alpha+i\beta)=F_1(\alpha+i\beta)\]
defines the first component $F_1$ as an even function on $D$.
\end{proof}

We also obtain the following \emph{representation formulas for slice functions}.

\begin{proposition}\label{pro5}
Let $f\in\Sl(\OO_D)$. Let $J,K\in\SS_A$ with $J-K$ invertible. Then the following formula holds:
\begin{equation}\label{rep1}
f(x)=(I-K)\left((J-K)^{-1}f(\alpha+\beta J)\right)-(I-J)\left((J-K)^{-1}f(\alpha+\beta K)\right)
\end{equation}
for every $I\in\SS_A$ and for every $x=\alpha+\beta I\in D_I=\OO_D\cap\C_I$. In particular, for $K=-J$, we get the formula
\begin{equation}\label{rep2}
f(x)=\frac12\left(f(\alpha+\beta J)+f(\alpha-\beta J)\right)-\frac{I}2\left(J\left(f(\alpha+\beta J)-f(\alpha-\beta J)\right)\right)
\end{equation}
\end{proposition}

\begin{proof}Let $z=\alpha+i\beta$.
From the proof of the preceding proposition, we get 
\[F_2(z)=(J-K)^{-1}\left(f(\alpha+\beta J)-f(\alpha+\beta K)\right)\text{\quad and}\]
\[F_1(z)=f(\alpha+\beta J)-JF_2(z)=f(\alpha+\beta J)-J((J-K)^{-1}\left(f(\alpha+\beta J)-f(\alpha+\beta K)\right).\]
Therefore
\begin{align*}
f(\alpha+\beta I)&=f(\alpha+\beta J)+(I-J)((J-K)^{-1}\left(f(\alpha+\beta J)-f(\alpha+\beta K)\right))=\\
&=((J-K)+(I-J))((J-K)^{-1}f(\alpha+\beta J))-(I-J)((J-K)^{-1}f(\alpha+\beta K))=\\
&=(I-K)\left((J-K)^{-1}f(\alpha+\beta J)\right)-(I-J)\left((J-K)^{-1}f(\alpha+\beta K)\right).
\end{align*}
We used the fact that, in any alternative algebra $A$, the equality $b=(aa^{-1})b=a(a^{-1}b)$ holds for every elements $a,b$ in $A$, $a$ invertible.
\end{proof}

Representation formulas for quaternionic Cullen regular functions appeared in \cite{CoGeSaPreprint2008} and \cite{CoGeSaSt2009Adv}. For slice monogenic functions of a Clifford variable, they were given in \cite{CoSaPreprint2008,CoSaTM2009}.
 
In the particular case when $I=J$, formula \eqref{rep2} reduces to the trivial decomposition
\[f(x)=\frac12\left(f(x)+f(x^c)\right)+\frac{1}2\left(f(x)-f(x^c)\right).\]
But $\frac12\left(f(x)+f(x^c)\right)=F_1(z)$ and $\frac{1}2\left(f(x)-f(x^c)\right)=JF_2(z)$, where $x=\alpha+\beta J$, $z=\alpha+i\beta$. 

\begin{definition}
Let $f\in\Sl(\OO_D)$. We call \emph{spherical value of $f$ in $x\in\OO_D$} the element of $A$
\[v_s f(x):=\frac12\left(f(x)+f(x^c)\right).\]
We call \emph{spherical derivative of $f$ in $x\in\OO_D\setminus\R$} the element of $A$
\[\partial_s f(x):=\frac12 \IM(x)^{-1}(f(x)-f(x^c)).\]
\end{definition}

In this way, we get two slice functions associated with $f$: $v_s f$ is induced on $\OO_D$ by the stem function $F_1(z)$ and $\partial_s f$ is induced on $\OO_D\setminus\R$ by ${F_2(z)}/{\IM(z)}$. Since these stem functions are $A$--valued, $v_s f$ and $\sd f$ are constant on every ``sphere'' 
\[\SS_x:=\{y\in Q_A\ |\ y=\alpha+\beta I,\ I\in\SS_A\}.\]
Therefore $\sd(\sd f)=0$ and $\sd(v_s f)=0$ for every $f$.
Moreover, $\sd f(x)=0$ if and only if $f$ is constant on $\SS_x$. In this case, $f$ has value $v_s f(x)$ on $\SS_x$.
If $\OO_D\cap\R\ne\emptyset$, under mild regularity conditions on $F$, we get that $\sd f$ can be continuously extended as a slice function on $\OO_D$. For example, it is sufficient to assume that $F_2(z)$ is of class $C^1$. By definition, the following identity holds for every $x\in\OO_D$:
\[f(x)=v_s f(x)+\IM(x)\,\sd f(x).\]

\begin{example}
Simple computations show that $v_s x^n=\frac12t(x^n)$ for every $n$, while
\[\sd x=1,\quad \sd x^2=x+x^c=t(x),\quad \sd x^3
=x^2+(x^c)^2+xx^c=t(x^2)+n(x).\] 
In general, the spherical derivative of a polynomial $\sum_{j=0}^n x^ja_j$ of degree $n$ is a real--valued polynomial in $x$ and $x^c$ of degree $n-1$. If $A_n$ and $B_n$ denote the real components of the complex power $z^n=(\alpha+i\beta)^n$, i.e.\ $z^n=A_n+iB_n$, then $B_n(\alpha,\beta)=\beta B'_n(\alpha,\beta^2)$, for a polynomial $B'_n$. Then 
\[\sd x^n=B'_n\left(\frac{x+x^c}2,-\left(\frac{x-x^c}2\right)^2\right).\]


\end{example}

\subsection{Smoothness of slice functions}
\label{sec:SmoothnessOfSliceFunctions}

Given a subset $B$ of $A$, a function $g:B\rightarrow A$ and a positive integer $s$ or $s\in\{\infty,\omega\}$, we will say that $g$ is of class $C^s(B)$ if $g$ can be extended to an open set as a function of class $C^s$ in the usual sense.

\begin{proposition}
Let $f=\I(F)\in\Sl(\OO_D)$ be a slice function. Then the following statements hold:
\begin{itemize}
\item[$(1)$]
If $F\in C^0(D)$, then $f\in C^0(\OO_D)$. Moreover, $v_s f\in  C^0(\OO_D)$ and $\sd f\in C^0(\OO_D\setminus\R)$.
\item[$(2)$]
If $F\in C^{2s+1}(D)$ for a positive integer $s$, then $f,v_s f$ and $\sd f$ are of class $C^s(\OO_D)$. As a consequence, if $F\in C^\infty(D)$, then the functions $f, v_s f, \sd f$ are of class $C^\infty(\OO_D)$. 
\item[$(3)$]
 If $F\in C^\omega(D)$, then  $f, v_s f$ and $\sd f$ are of class $C^\omega(\OO_D)$.
\end{itemize}
\end{proposition}
\begin{proof}
Assume that $F\in C^0(D)$. For any $x\in Q_A$, let $x=\RE(x)+\IM(x)$ be the decomposition given in Proposition~\ref{pro2}. Since $\IM(x)$ is purely imaginary,
$\IM(x)^2$ is a  non--positive real number. The map that associates $x\in Q_A$ with $z=\RE(x)+i\sqrt{-\IM(x)^2}\in\C$ is continuous on $Q_A\setminus\R$. The same property holds for the map that sends $x$ to $J=\IM(x)/\sqrt{n(\IM(x))}$. These two facts imply that on $\OO_D\setminus\R$ the slice function $f=\I(F)$ has the same smoothness as the restriction of the stem function $F$ on $D\setminus\R$.
Continuity at real points of $\OO_D$ is an immediate consequence of the even--odd character of the pair $(F_1(z),F_2(z))$ w.r.t.\ the imaginary part of $z$.

If $F$ is $C^{2s+1}$--smooth or $C^\omega$ and $z=\alpha+i\beta$, it follows from a result of Whitney \cite{Whitney1943} that there exist  $F'_1,F_2'$ of class $C^s$ or $C^\omega$ respectively, such that
\[F_1(\alpha,\beta)= F'_1(\alpha,\beta^2),\ F_2(\alpha,\beta)=\beta  F_2'(\alpha,\beta^2)\text{\quad on $D$}.\]
(Here $F_j(\alpha,\beta)$ means $F_j(z)$ as a function of $\alpha$ and $\beta$, $j=1,2$.)
Then, for $x=\RE(x)+\IM(x)\in Q_A\cap\C_J$, $\RE(x)=\alpha$, $\IM(x)=J\beta$, $\beta^2=-\IM(x)^2=n(\IM(x))$, it holds:
\begin{align*}
v_sf(x)&=F'_1(\RE(x),n(\IM(x)))=F'_1\left(\frac{x+x^c}2,n\left(\frac{x-x^c}2\right)\right),\\
\sd f(x)&=F'_2(\RE(x),n(\IM(x)))=F'_2\left(\frac{x+x^c}2,n\left(\frac{x-x^c}2\right)\right)
\end{align*}
 and $f(x)=v_s f(x)+\IM(x)\,\sd f(x)=v_s f(x)+\frac12(x-x^c)\,\sd f(x)$.
These formulas imply the second and third statements.
\end{proof}

We will denote by
\[\Sl^1(\OO_D):=\{f=\I(F)\in\Sl(\OO_D)\ |\ F\in C^1(D)\}\]
the real vector space of slice functions with stem function of class $C^1$.

Let $f=\I(F)\in\Sl^1(\OO_D)$ and $z=\alpha+i\beta\in D$. Then the partial derivatives $\dt{F}{\alpha}$ and $i\dt{F}{\beta}$ are continuous $A$--stem functions on $D$. The same property holds for their linear combinations \[\dd{F}{z}=\frac12\left(\dd{F}{\alpha}-i\dd{F}{\beta}\right)\text{\quad and\quad}\dd{F}{\bar z}=\frac12\left(\dd{F}{\alpha}+i\dd{F}{\beta}\right).\]

\begin{definition}\label{def6}
Let $f=\I(F)\in\Sl^1(\OO_D)$. We set 
\[\dd{f}{x}:=\I\left(\dd{F}{z}\right),\quad \dd{f}{x^c}:=\I\left(\dd{F}{\bar z}\right).\]
These functions are continuous slice functions on $\OO_D$. 
\end{definition}

The notation $\dt{f}{x^c}$ is justified by the following properties: $x^c=\I(\bar z)$ and therefore $\dt{x^c}{x^c}=1$, $\dt{x}{x^c}=0$.

\section[Slice regular functions]{Slice regular functions}
\label{sec:SliceRegularFunctions}

Left multiplication by $i$ defines a complex structure on $A_\C$. With respect to this structure, a $C^1$ function
$F=F_1+iF_2:D\rightarrow A_{\C}$ is holomorphic if and only if its components $F_1,F_2$ satisfy the Cauchy--Riemann equations:
\[\dd{F_1}{\alpha}=\dd{F_2}{\beta},\quad \dd{F_1}{\beta}=-\dd{F_2}{\alpha}\quad(z=\alpha+i\beta\in D),\quad \text{i.e.}\quad \dd{F}{\bar z}=0.\]

This condition is equivalent to require that, for any basis $\B$, the complex curve $\widetilde F_\B$ (cf.\ Remark~\ref{rem1} in Subsection~\ref{sec:AstemFunctions}) is holomorphic.

\begin{definition}
A (left) slice function $f\in\Sl^1(\OO_D)$ is \emph{(left) slice regular} if its stem function $F$ is holomorphic. We will denote the vector space of slice regular functions on $\OO_D$ by
\[\SR:=\{f\in\Sl^1(\OO_D)\ |\ f=\I(F),\ F:D\rightarrow \Ac \text{ holomorphic}\}.\]
\end{definition}

\begin{remarks}\label{rem2}
$(1)$
A function $f\in\Sl^1(\OO_D)$ is slice regular if and only if the slice function $\dt{f}{x^c}$ (cf.\ Definition~\ref{def6}) vanishes identically. Moreover, if $f$ is slice regular, then also $\dt{f}{x}=\I\left(\dt{F}{z}\right)$ is slice regular on $\OO_D$.

$(2)$ 
Since $v_s f$ and $\sd f$ are $A$--valued, they are slice regular only when they are locally constant functions.
\end{remarks}

As seen in the Examples~\ref{ex2} of Subsection~\ref{sec:Avaluedslicefunctions}, polynomials with right coefficients belonging to $A$ can be considered as slice regular functions on the quadratic cone.  If $f(x)=\sum_{j=0}^m x^ja_j$, then $f=\I(F)$, with $F(z)=\sum_{j=0}^m z^ja_j$. 

Assume that on $A$ is defined a positive scalar product $x\cdot y$ whose associated norm satisfies an inequality $|xy|\le C|x||y|$ for a positive constant $C$ and such that $|x|^2=n(x)$ for every $x\in Q_A$. Then we also have 
$|x^k|=|x|^k$ for every $k>0$, $x\in Q_A$. In this case we can consider also convergent power series $\sum_k x^k a_k$ as slice regular functions on the intersection of the quadratic cone with a ball (w.r.t.\ the norm $|x|$) centered in the origin. See for example \cite[\S4.2]{GHS} for the quaternionic and Clifford algebra cases, where  we can take as product $x\cdot y$ the euclidean product in $\R^4$ or $\R^{2^n}$, respectively.

\begin{proposition}\label{pro7}
Let $f=\I(F)\in \Sl^1(\OO_D)$.
Then $f$ is slice regular on $\OO_D$ if and only if the restriction 
\[f_J:=f\,_{|\C_J\cap\OO_D}\ :\ D_J=\C_J\cap\OO_D\rightarrow A\]
 is holomorphic for every $J\in\SS_A$ with respect to the complex structures on $D_J$ and $A$ defined by left multiplication by $J$.
\end{proposition}
\begin{proof}
Since
$f_J(\alpha+\beta J)=F_1(\alpha+i\beta)+JF_2(\alpha+i\beta)$, if $F$ is holomorphic then
\[\dd{f_J}{\alpha}+J\dd{f_J}{\beta}=\dd{F_1}{\alpha}+J\dd{F_2}{\alpha}+J\left(\dd{F_1}{\beta}+J\dd{F_2}{\beta}\right)=0\]
at every point $x=\alpha+\beta J\in D_J$.
Conversely, assume that $f_J$ is holomorphic for every $J\in\SS_A$. Then
\[0=\dd{f_J}{\alpha}+J\dd{f_J}{\beta}=\dd{F_1}{\alpha}-\dd{F_2}{\beta}+J\left(\dd{F_2}{\alpha}+\dd{F_1}{\beta}\right)
\]
at every point $z=\alpha+i\beta\in D$. From the arbitrariness of $J$ it follows that $F_1,F_2$ satisfy the Cauchy--Riemann equations.
\end{proof}

\begin{remark}
The even--odd character of the pair $(F_1,F_2)$ and the proof of the preceding proposition show that, in order to get slice regularity of $f=\I(F)$, $F\in C^1$, it is sufficient to assume that two functions $f_J$, $f_K$ ($J-K$ invertible) are holomorphic on domains $\C^+_J\cap\OO_D$ and $\C^+_K\cap\OO_D$, respectively (cf.\ Proposition~\ref{pro4} for notations). The possibility $K=-J$ is not excluded: it means that the single function $f_J$ must be holomorphic on $D_J$.
\end{remark}

Proposition~\ref{pro7} implies that if  $A$ is the algebra of quaternions or octonions, and the domain $D$ intersects the real axis, then $f$ is slice regular on $\OO_D$ if and only if it is \emph{Cullen regular} in the sense introduced by Gentili and Struppa in \cite{GeSt2006CR,GeSt2007Adv} for quaternionic functions and in \cite{GeStRocky,GeStoStVl} for octonionic functions.


If $A$ is the real Clifford algebra $\R_n$, slice regularity generalizes the concept of \emph{slice monogenic functions} introduced by Colombo, Sabadini and Struppa in \cite{CoSaSt2009Israel}. From Proposition~\ref{pro7} and from the inclusion of the set $\R^{n+1}$ of paravectors in the quadratic cone (cf.\ Example~\ref{ex1} of Section~\ref{sec:TheQuadraticCone}), we get the following result.

\begin{corollary}\label{cor8}
Let $A=\R_n$ be the Clifford algebra. If $f=\I(F)\in\SR$, $F\in C^1(D)$ and $D$ intersects the real axis, then the restriction of $f$ to the subspace of paravectors is a slice monogenic function
\[f_{|\OO_D\cap \R^{n+1}}\ :\ {\OO_D\cap \R^{n+1}}\rightarrow \R_n.\]
\end{corollary}

Note that every slice monogenic function is the  restriction to the subspace of paravectors of a unique slice regular function (cf.\ Propositions~\ref{pro4} and~\ref{pro5}).

\section{Product of slice functions}
\label{sec:Productofslicefunctions}

In general, the pointwise product of two slice functions is not a slice function. However, pointwise product in the algebra $\Ac$ of $A$--stem functions induces a natural product on slice functions.
 
\begin{definition}
Let $f=\I(F),g=\I(G)\in\Sl(\OO_D)$. The \emph{product} of $f$ and $g$ is the slice function
\[f\cdot g:=\I(FG)\in\Sl(\OO_D).\]
\end{definition}
The preceding definition is well--posed, since the pointwise product $FG=(F_1+iF_2)(G_1+iG_2)=F_1G_1-F_2G_2+i(F_1G_2+F_2G_1)$ of complex intrinsic functions is still complex intrinsic.
It follows directly from the definition that the product is distributive. It is also associative if $A$ is an associative algebra.

The spherical derivative satisfies a Leibniz--type product rule, where evaluation is replaced by spherical value:
\[\sd (f\cdot g)=(\sd f)(v_s g)+(v_s f)(\sd g).\]

\begin{remark}
In general, $(f\cdot g)(x)\ne f(x)g(x)$. If $x=\alpha+\beta J$ belongs to $D_J=\OO_D\cap\C_J$ and $z=\alpha+i\beta$, then 
\[(f\cdot g)(x)=F_1(z)G_1(z)-F_2(z)G_2(z)+J\left(F_1(z)G_2(z)\right)+J\left(F_2(z)G_1(z)\right),\]
while
\[f(x)g(x)=F_1(z)G_1(z)+(JF_2(z))(JG_2(z))+F_1(z)(JG_2(z))+(JF_2(z))G_1(z).\]
If the components $F_1,F_2$ of the \emph{first} stem function $F$ are real--valued,  or if $F$ and $G$ are both $A$--valued, then $(f\cdot g)(x)= f(x)g(x)$ for every $x\in \OO_D$. In this case, we will use also the notation $fg$ in place of $f\cdot g$.
\end{remark}

\begin{definition}
A slice function $f=\I(F)$ is called \emph{real} if the $A$--valued components $F_1,F_2$ of its stem function are real--valued. Equivalently, $f$ is real if the spherical value $v_s f$ and the spherical derivative $\sd f$ are real--valued.
\end{definition}

A real slice function $f$ has the following property: for every $J\in\SS_A$, the image $f(\C_J\cap\OO_D)$ is contained in $\C_J$. More precisely, this condition characterizes the reality of $f$.

\begin{proposition}
A slice function $f\in\Sl(\OO_D)$ is real if and only if\, $f(\C_J\cap\OO_D)\subseteq\C_J$ for every $J\in\SS_A$.
\end{proposition}
\begin{proof}Assume that $f(\C_J\cap\OO_D)\subseteq\C_J$ for every $J\in\SS_A$.
Let $f=\I(F)$. If $x=\alpha+\beta J\in\OO_D$ and $z=\alpha+i\beta$, then 
$f(x)=F_1(z)+JF_2(z)\in\C_J$ and $f(x^c)=f(\alpha-\beta J)=F_1(\bar z)+J F_2(\bar z)=F_1(z)-JF_2(z)\in\C_J$. This implies that $F_1(z), F_2(z)\in\bigcap_J\C_J=\R$ (cf.\ Proposition~\ref{pro3}).
\end{proof}

\begin{proposition} \label{pro9}
If $f,g$ are slice regular on $\OO_D$, then the product $f\cdot g$ is slice regular on $\OO_D$.
\end{proposition}
\begin{proof}
Let $f=\I(F)$, $g=\I(G)$, $H=FG$. If $F$ and $G$ satisfy the Cauchy--Riemann equations, the same holds for $H$. This follows from the validity of the Leibniz product rule, that can be checked using a basis representation of $F$ and $G$.
\end{proof}

Let $f(x)=\sum_jx^ja_j$ and $g(x)=\sum_kx^kb_k$ be polynomials or convergent power series with coefficients $a_j,b_k\in A$.
The usual product of polynomials, where $x$ is considered to be a commuting variable (cf.\ for example \cite{Lam} and \cite{GelfandRW2001,GelfandGRW2005}), can be extended to power series (cf.\ \cite{GeSto2008Mich,GeSt2008Milan} for the quaternionic case) in the following way: the \emph{star product} $f \ast g$ of $f$ and $g$ is the convergent power series defined by setting
\[\textstyle(f \ast g)(x):=\sum_nx^n\big(\sum_{j+k=n}a_jb_k\big).\]

\begin{proposition}
Let  $f(x)=\sum_jx^ja_j$ and $g(x)=\sum_kx^kb_k$ be polynomials or convergent power series ($a_j,b_k\in A$). Then the product of $f$ and $g$, viewed as slice regular functions, coincides with the {star product} $f \ast g$, i.e.\ $\I(FG)=\I(F)\ast\I(G)$.
\end{proposition}
\begin{proof}
Let $f=\I(F)$, $g=\I(G)$, $H=FG$. Denote by $A_n(z)$ and $B_n(z)$ the real components of the complex power $z^n=(\alpha+i\beta)^n$. Since $\R\otimes_{\R}\C\simeq\C$ is contained in the commutative and associative center of $\Ac$, we have
\[\textstyle
H(z)=F(z)G(z)=\big(\sum_jz^ja_j\big)\big(\sum_kz^kb_k\big)=\sum_{j,k}z^jz^k(a_jb_k).
\]
Let $c_n=\sum_{j+k=n}a_jb_k$ for each $n$. Therefore, we have
\begin{align*}
\textstyle H(z)&=\sum_{n}z^nc_n=\sum_n (A_n(z)+iB_n(z))c_n=\sum_n A_n(z)c_n+i\big(\sum_nB_n(z)c_n\big)=\\
&=:H_1(z)+iH_2(z)
\end{align*}
and then, if $x=\alpha+\beta J$ and $z=\alpha+i\beta$,
\[\textstyle
\I(H)(x)=H_1(z)+JH_2(z)=\sum_n A_n(z)c_n+J\big(\sum_nB_n(z)c_n\big).\]
On the other hand, $(f\ast g)(x)=\sum_n(\alpha+\beta J)^nc_n=\sum_n(A_n(z)+JB_n(z))c_n$ and the result follows.
\end{proof}

\begin{example}
Let $A=\HH$, $I,J\in\SS_\HH$. Let $F(z)=z-I$, $G(z)=z-J$. 
Then $f(x)=x-I$, $g(x)=x-J$,  $(FG)(z)=z^2-z(I+J)+IJ$ and  $(f\cdot g)(x)=(x-I)\cdot(x-J)=\I(FG)=x^2-x(I+J)+IJ=\I(F)\ast\I(G)$.

Note that $(f\cdot g)(x)$ is different from $f(x)g(x)=(x-I)(x-J)=x^2-xJ-Ix+IJ$ for every $x$ not lying in $\C_I$.
\end{example}

\section{Normal function and admissibility}
\label{sec:NormalFunctionAndAdmissibility}

We now associate to every slice function a new slice function, the \emph{normal function}, which will be useful in the following sections when dealing with zero sets. Our definition is equivalent to the one given in \cite{GeSto2008Mich} for (Cullen regular) quaternionic power series.
There the normal function was named the \emph{symmetrization} of the power series.

\begin{definition}
Let $f=\I(F)\in\Sl(\OO_D)$. Then also $F^c(z):=F(z)^c={F_1(z)}^c+i{F_2(z)}^c$ is an $A$--stem function. We set:
\begin{itemize}
\item
$f^c:=\I(F^c)\in\Sl(\OO_D)$.
\item
$CN(F):=FF^c=F_1{F_1}^c-F_2{F_2}^c+i(F_1{F_2}^c+F_2{F_1}^c)=n(F_1)-n(F_2)+i\,t(F_1{F_2}^c)$.
\item
$N(f):=f\cdot f^c=\I(CN(F))\in\Sl(\OO_D)$.
\end{itemize}
The slice function $N(f)$ will be called  the \emph{normal function} of $f$.
\end{definition}

\begin{remarks}
$(1)$
Partial derivatives commute with the antiinvolution $x\mapsto x^c$. From this fact and from  Proposition~\ref{pro9}, it follows that, if $f$ is slice regular, then also $f^c$ and $N(f)$ are slice regular. 

$(2)$ 
Since $x\mapsto x^c$ is an antiinvolution, $(FG)^c=G^cF^c$ and then $(f\cdot g)^c=g^c\cdot f^c$. Moreover, $N(f)=N(f)^c$, while $N(f^c)\ne N(f)$ in general.

$(3)$ 
For every slice functions $f,g$, we have $v_s(f^c)=(v_s f)^c$, $\sd(f^c)=(\sd f)^c$
and $\sd N(f)=(\sd f)(v_s f^c)+(v_s f) (\sd f^c)$.
\end{remarks}

If $A$ is the algebra of quaternions or octonions, then $CN(F)$ is complex--valued and then the normal function $N(f)$ is real. For a general algebra $A$, this is not true for every slice function. This is the motivation for the following definition.

\begin{definition}
A slice function $f=\I(F)\in\Sl(\OO_D)$ is called \emph{admissible} if the spherical value of $f$ at $x$ belongs to the normal cone $N_A$  for every $x\in\OO_D$ and the real vector subspace $\langle v_s f(x),\sd f(x)\rangle$ of $A$, generated by the spherical value and the spherical derivative at $x$, is contained in the normal cone $N_A$ for every $x\in\OO_D\setminus\R$.
Equivalently, 
\[\langle F_1(z),F_2(z)\rangle\subseteq N_A\text{\quad for every }z\in D.\]
\end{definition}

\begin{remarks}
$(1)$
If $A=\HH$ or $\Oc$, then \emph{every} slice function is admissible, since $N_A=Q_A=A$.

$(2)$ 
	If  $f$ is admissible, then $CN(F)$ is complex--valued and then $N(f)$ is real. Indeed, if $F_1(z),F_2(z)$ and $F_1(z)+F_2(z)$ belong to $N_A$, $n(F_1(z))$ and $n(F_2(z))$ are real and $t(F_1(z)F_2(z)^c)=n(F_1(z)+F_2(z))-n(F_1(z))-n(F_2(z))$ is real.

$(3)$ 
If $f$ is real, $f^c=f$, 
	$N(f)=f^2$ and $f$ is admissible.

$(4)$ 
	If $f$ is real and $g$ is admissible, also $fg$ is admissible.
\end{remarks}

\begin{proposition}\label{pro12}
If $n(x)=n(x^c)\ne0$ for every $x\in A\setminus\{0\}$ such that $n(x)$ is real, then a slice function $f$ is admissible if and only if $N(f)$ and $N(\sd f)$ are real.
\end{proposition}

\begin{proof}
$N(\sd f)$ is real if and only if $n(F_2(z))\in\R$ for every $z\in D\setminus\R$. Assume that also $N(f)$ is real. Then $n(F_1(z))$ and $n(F_1(z)+F_2(z))$ are real, from which it follows that $n(\alpha F_1(z)+\beta F_2(z))$ is real for every $\alpha,\beta\in\R$. Since $n(x)$ real implies $x\in N_A$, we get that $\langle F_1(z),F_2(z)\rangle\subseteq N_A$ for every $z\in D\setminus\R$. If $z$ is real, then $N(f)(z)=n(F_1(z))$ is real and $F_2(z)=0$. This means that the condition of admissibility is satisfied for every $z\in D$.
\end{proof}

\begin{example}
Consider the Clifford algebra $\R_3$ with the usual conjugation. Its normal cone   
\[N_{\R_3}=\{x\in\R_3\ |\ x_0x_{123}-x_1x_{23}+x_2x_{13}-x_3x_{12}=0\}\]
contains the subspace $\R^4$ of paravectors (cf.\ Examples~\ref{ex1} of Section~\ref{sec:TheQuadraticCone}). Every polynomial $p(x)=\sum_nx^na_n$ with paravectors coefficients $a_n\in\R^4$ is an admissible slice regular function on $Q_{\R_3}$, since $P(z)=\sum_nz^na_n=\sum A_n(z)a_n+i(\sum B_n(z)a_n)$ belongs to $\R^4\otimes_\R\C$ for every $z\in\C$. 

The polynomial $p(x)=xe_{23}+e_1$ is an example of a non--admissible slice regular function on $\R_3$, since $\sd p(x)=e_{23}\in N_A$ for every $x$, but $v_s p(x)=\RE(x)e_{23}+e_1\notin N_A$ if $\RE(x)\ne0$.
\end{example}

\begin{theorem}\label{thmNfg}
Let $A$ be associative or $A=\Oc$. Then \[N(f\cdot g)=N(f)\, N(g)\]
for every admissible $f,g\in\Sl(\OO_D)$. 
\end{theorem}
\begin{proof}
Assume that $A$ is associative. Let $f=\I(F)$, $g=\I(G)$ be admissible and $z\in D$. Denote the norm of $w\in\Ac$ by $cn(w):=ww^c=n(x)-n(y)+i\,t(xy^c)\in\Ac$. 
Then $CN(FG)(z)=cn(F(z)G(z))=(F(z)G(z))(F(z)G(z))^c=cn(F(z))\,cn(G(z))$, since  $cn(F(z))$ and $cn(G(z))$ are in the center of $\Ac$. Then \[N(f\cdot g)=\I(CN(FG))=\I(CN(F)\, CN(G))=N(f)\, N(g).\]

If $A=\HH$ or $\Oc$, a different proof can be given. 
$A_\C$ is a complex alternative algebra with an antiinvolution $x\mapsto x^c$	such that $x+x^c,xx^c\in\C\ \forall x$. Then $A_\C$ is an \emph{algebra with composition} (cf.\ \cite[p.\ 58]{Schafer}): the {complex norm} $cn(x)=xx^c$ is multiplicative (for $\Oc\otimes_\R\C$ it follows from Artin's Theorem) and we can conclude as before.
\end{proof}

The multiplicativity of the normal function of octonionic power series was already proved in \cite{GhPePreprint2009} by a direct computation.

\begin{corollary}\label{cor15}
Let $A$ be associative or $A=\Oc$.  Assume that $n(x)=n(x^c)\ne0$ for every $x\in A\setminus\{0\}$ such that $n(x)$ is real. If $f$ and $g$ are admissible slice functions, then also the product $f\cdot g$ is admissible.
\end{corollary}
\begin{proof}
We apply Proposition~\ref{pro12}. If $N(f)$ and $N(g)$ are real, then $N(f\cdot g)=N(f) N(g)$ is real. Now consider the spherical derivatives: $N(\sd (f\cdot g))=N((\sd f)(v_s g))+N((v_s f)(\sd g))=N(\sd f)N(v_s g)+N(v_s f)N(\sd g)$ is real, since all the slice functions are real.
\end{proof}

\begin{example}\label{ex4}
Let $A=\R_3$. This algebra satisfies the condition of the preceding corollary: if $n(x)$ is real, then $n(x)=n(x^c)$ (cf.\ e.g.\ \cite{GHS}). Consider the admissible polynomials  $f(x)=xe_2+e_1$, $g(x)=xe_3+e_2$. Then  $(f\cdot g)(x)=x^2e_{23}+x(e_{13}-1)+e_{12}$ is admissible, $N(f)=N(g)=x^2+1$ and $N(f\cdot g)=(x^2+1)^2$.
\end{example}

\section{Zeros of slice functions}
\label{sec:ZerosOfSliceFunctions}

The zero set $V(f)=\{x\in Q_A\ |\ f(x)=0\}$ of an admissible slice function $f\in\Sl(\OO_D)$ has a particular structure. We will see in this section that, for every fixed $x=\alpha+\beta J\in Q_A$, the ``sphere'' 
\[\SS_x=\{y\in Q_A\ |\ y=\alpha+\beta I,\ I\in\SS_A\}\]
is entirely contained in $V(f)$ or it contains at most one zero of $f$. Moreover, if $f$ is not real, there can be isolated, non--real zeros. These different types of zeros of a slice function correspond, at the level of the stem function, to the existence of zero--divisors in the complexified algebra $\Ac$.

\begin{proposition}
Let $f\in\Sl(\OO_D)$. If the spherical derivative of $f$ at $x\in\OO_D\setminus\R$ belongs to $N_A$, then the restriction of $f$ to $\SS_x$ is injective or constant. In particular, either $\SS_x\subseteq V(f)$  or $\SS_x\cap V(f)$ consists of a single point.
\end{proposition}
\begin{proof}
If $\sd f(x)\in N_A$, then it is not a zero--divisor of $A$. Given $x,x'\in\SS_x$, if $f(x)=f(x')$, then $(\IM(x)-\IM(x'))\,\sd f(x)=(x-x')\,\sd f(x)=0$.  If $\sd f(x)\ne0$, this implies $x=x'$.
\end{proof}

\begin{remark}
The same conclusion of the preceding proposition holds for \emph{any} slice function when the  function is restricted to a subset of $\SS_x$ that does not contain pairs of points $x,x'$ such that $x-x'$ is a left zero--divisor in $A$. This is the case e.g.\ of the slice monogenic functions, which are defined on the paravector space of a Clifford algebra (cf.\ Corollary~\ref{cor8} and \cite{CoSaSt2009Israel}).
\end{remark}

 \begin{theorem}[Structure of $V(f)$]
 \label{structure_thm}
 Let $f=\I(F)\in\Sl(\OO_D)$. 
 Let $x=\alpha+\beta J\in\OO_D$ and $z=\alpha+i\beta\in D$. Assume that, if $x\notin\R$, then $\sd f(x)\in N_A$. Then one of the following mutually exclusive statements holds:
 \begin{itemize}
 \item[$(1)$]
 $\SS_x\cap V(f)=\emptyset$. 
 \item[$(2)$]
$\SS_x\subseteq V(f)$.
 In this case $x$ is called a \emph{real} (if $x\in\R$) or \emph{spherical}  (if $x\notin\R$) \emph{zero} of $f$.
  \item[$(3)$]
 $\SS_x\cap V(f)$ consists of a single, non--real point. In this case $x$ is called an \emph{$\SS_A$--isolated non--real zero} of $f$.
 \end{itemize}
These three possibilities correspond, respectively, to the following properties of $F(z)\in\Ac$:
\begin{itemize}
\item[$(1')$] $CN(F)(z)=F(z)F(z)^c\ne0$.
\item[$(2')$] $F(z)=0$.
\item[$(3')$] $F(z)\ne0$ and $CN(F)(z)=0$ (in this case $F(z)$ is a zero--divisor of $\Ac$).
\end{itemize}
\end{theorem}

Before proving the theorem, we collect some algebraic results in the following lemma.

\begin{lemma}\label{alg_lemma}
Let $w=x+iy\in\Ac$, with $y\in N_A$. Let $cn(w):=ww^c=n(x)-n(y)+i\,t(xy^c)\in\Ac$.  Then the following statements hold:
\begin{itemize}
\item[$(1)$] $cn(w)=0$ if and only if $w=0$ or there exists a unique $K\in\SS_A$ such that $x+Ky=0$.
\item[$(2)$] If  $\langle x,y\rangle\subseteq N_A$, then $cn(w)\in\C$. Moreover, $cn(w)\ne0$ if and only if $w$ is invertible in $\Ac$.
\end{itemize}
\end{lemma}

\begin{proof}
In the proof we use that $(x,y^{-1},y)=0$ for every $x,y$, with $y$ invertible.

If $cn(w)=0$, $w\ne0$, then $n(x)=n(y)\ne0$ and $y$ is invertible, with inverse $y^c/n(y)$. Moreover, $(z,y^c,y)=0$ for every $z$. Set $K:=-xy^{-1}$. Then $Ky=-(xy^{-1})y=-x(y^{-1}y)=-x$. Moreover, from $t(xy^c)=xy^c+yx^c=0$ it follows that
\[K^c=-\frac{yx^c}{n(y)}=\frac{xy^c}{n(y)}=-K\text{\quad and}\]
\[K^2=\left(-\frac{xy^c}{n(y)}\right)\left(-\frac{xy^c}{n(y)}\right)=-\frac{(xy^c)(yx^c)}{n(x)n(y)}=-\frac{x(y^cy)x^c}{n(x)n(y)}=-1.\]
Then $K\in\SS_A$. Uniqueness of $K\in\SS_A$ such that $x+Ky=0$, comes immediately from the invertibility of $y\in N_A$, $y\ne0$. Conversely, if $x+Ky=0$, then $n(x)=(-Ky)(y^cK)=-K^2n(y)=n(y)$ and $xy^c+yx^c=-(Ky)y^c+y(y^cK)=0$. Therefore $cn(w)=0$.

If $x,y$ and $x+y$ belong to $N_A$, $n(x+y)-n(x)-n(y)=t(xy^c)$ is real, which implies that $cn(w)$ is complex. Let $cn(w)=:u+iv\ne0$, with $u,v$ real. Then $w':=(x^c+iy^c)(u-iv)/(u^2+v^2)$ is the inverse of $w$. On the other hand, if $cn(w)=0$, $w$ is a divisor of zero of $\Ac$.
\end{proof}

\begin{proof}[Proof of Theorem~\ref{structure_thm}]
If $x=\alpha$ is real, then $\SS_x=\{x\}$ and $f(x)=F_1(\alpha)=0$ if and only if $F(\alpha)=0$, since $F_2$ vanishes on the real axis. Therefore $f(x)=0$ is equivalent to $F(z)=0$. 

Now assume that $x\notin\R$. If $F(z)=0$, then $f(\alpha+\beta I)=F_1(z)+IF_2(z)=0$ for every $I\in\SS_A$. Then $\SS_x\subseteq V(f)$ (a spherical zero). 
Since $\sd f(x)\in N_A$, also $F_2(z)\in N_A$ and Lemma~\ref{alg_lemma} can be applied to $w=F(z)$. If $F(z)\ne0$ and $CN(F)(z)=0$, there exists a unique $K\in\SS_A$ such that $F_1(z)+KF_2(z)=0$, i.e.\ $f(\alpha+\beta K)=0$. Therefore $\SS_x\cap V(f)=\{\alpha+\beta K\}$. The last case to consider is $CN(F)(z)\ne0$. From Lemma~\ref{alg_lemma} we get that $f(\alpha+\beta I)=F_1(z)+IF_2(z)\ne0$ for every $I\in\SS_A$, which means that $\SS_x\cap V(f)=\emptyset$.  
\end{proof}

\begin{remark}
From the preceding proofs, we get that an $\SS_A$--isolated non--real zero $x$ of $f$ is given by the formula $x=\alpha+\beta K$, with $K:=-F_1(z){F_2(z)}^c/{n(F_2(z))}$, $CN(F)(z)=0$. This formula can be rewritten in a form that resembles Newton's method for finding roots:
\[x=\RE(x)-v_s f(x)\, (\sd f(x))^{-1}.\]
\end{remark}

\begin{corollary}\label{cor13}
It holds:
\begin{itemize}
\item[$(1)$] A real slice function has no $\SS_A$--isolated non--real zeros.
\item[$(2)$] For every admissible slice function $f$, we have
 \[V(N(f))=\bigcup_{x \in V(f)}\SS_x.\]
 \end{itemize}
\end{corollary}

\begin{proof}
If $f=\I(F)$ is real, then $CN(F)=F^2$. Therefore the third case of the theorem is excluded. If $f$ is admissible, then Theorem~\ref{structure_thm} can be applied.
If $x\in V(f)$, $x=\alpha+\beta J$, the proposition gives $CN(F)(z)=0$ ($z=\alpha+i\beta$) and when applied to 
$N(f)$ tells that $\SS_x\subseteq V(N(f))$. Conversely, since $N(f)$ is real, $N(f)(x)=0$ implies $0=CN(CN(F))(z)=CN(F)(z)^2$, i.e.\ $CN(F)(z)=0$. From the proposition applied to $f$ we get at least one point $y\in\SS_x\cap V(f)$. Then $x\in\SS_y$, $y\in V(f)$.
\end{proof}

\begin{theorem}
Let $\OO_D$ be connected.  If $f$ is slice regular and admissible on $\OO_D$, and $N(f)$  does not vanish identically, then  $\C_J\cap\bigcup_{x \in V(f)}\SS_x$ is closed and discrete in $D_J=\C_J\cap\OO_D$ for every $J\in\SS_A$. If $\OO_D\cap\R\ne\emptyset$, then $N(f)\equiv0$ if and only if $f\equiv0$.
\end{theorem}

\begin{proof}
The normal function $N(f)$ is a real slice regular function on $\OO_D$. For every $J\in\SS_A$, the restriction $N(f)_J:D_J\rightarrow\C_J$ is a holomorphic function, not identically zero (otherwise it would be $N(f)\equiv0$). Therefore its zero set 
\[\C_J\cap V(N(f))=\C_J\cap \bigcup_{x \in V(f)}\SS_x \]
is closed and discrete in $D_J$.
If there exists $x\in\OO_D\cap\R$ and $N(f)(x)=n(F_1(x))=0$, then $F(x)=F_1(x)=0$. Since $F$ is holomorphic, it can vanish on $\OO_D\cap\R$ only if $f\equiv0$ on $D$.
\end{proof}

In the quaternionic case, the structure theorem for the zero set of slice regular functions was proved by Pogorui and Shapiro \cite{PogoruiShapiro} for polynomials and by  Gentili and Stoppato \cite{GeSto2008Mich} for power series. See also \cite{CoSaStPreprintIsrael} for a similar result about slice monogenic functions.
 
\begin{remark}\label{rem8}
If $\OO_D$ does not intersect the real axis, a not identically zero slice regular function $f$ can have normal function $N(f)\equiv0$. For example, let $J\in\SS_\HH$ be fixed. The admissible slice regular function defined on $\HH\setminus\R$ by
\[f(x)=1-IJ\quad \text{($x=\alpha+\beta I\in\C_I^+$)}\]
is induced by a locally constant stem function and has zero normal function. Its zero set $V(f)$ is the half plane $\C_{-J}^+\setminus\R$. The function $f$ can be obtained by the representation formula \eqref{rep2} in Proposition \ref{pro5} by choosing the constant values 2 on $\C_J^+\setminus\R$ and $0$ on $\C_{-J}^+\setminus\R$.
\end{remark}

If an admissible slice function $f$ has $N(f)\not\equiv0$, then the $\SS_A$--isolated non--real zeros are genuine isolated points of $V(f)$ in $\OO_D$. In this case, 
$V(f)$ is a union of isolated ``spheres'' $\SS_x$ and isolated points.

\subsection{The Remainder Theorem}
\label{sec:TheRemainderTheorem}
In this section, we prove a division theorem, which generalizes a result  proved by Beck \cite{Beck1979} for quaternionic polynomials and by Ser\^odio \cite{Serodio2007} for octonionic polynomials.

\begin{definition}
For any $y\in Q_A$, the \emph{characteristic polynomial} of $y$  
 is the slice regular function on $Q_A$
\[\Delta_{y}(x):=N(x-y)=(x-y)\cdot (x-y^c)=x^2-x\,t(y)+n(y).\]
\end{definition}

\begin{proposition}
The characteristic polynomial $\Delta_{y}$ of $y\in Q_A$ is real. Two characteristic polynomials  $\Delta_{y}$ and $\Delta_{y'}$ coincides if and only if\, $\SS_{y}=\SS_{y'}$. Moreover,  $V(\Delta_{y})=\SS_{y}$.
\end{proposition}
\begin{proof}
The first property comes from the definition of the quadratic cone. Since $n(x)=n(\RE(x))+n(\IM(x))$ for every $x\in Q_A$, two elements $y,y'\in Q_A$ have equal trace and norm if and only if $\RE(y)=\RE(y')$ and $n(\IM(y))=n(\IM(y'))$.
Since $\SS_y$ is completely determined by its ``center'' $\alpha=\RE(y)$ and its ``squared radius'' $\beta^2=n(\IM(y))$,
$\Delta_{y}=\Delta_{y'}$ if and only if\, $\SS_{y}=\SS_{y'}$. The last property follows from $\Delta_y(y)=0$ and the reality of $\Delta_y$.
\end{proof}


\begin{theorem}[Remainder Theorem]\label{Remainder_Theorem}
Let $f\in\SR$ be an admissible slice regular function. 
Let $y\in V(f)=\{x\in Q_A\ |\ f(x)=0\}$. Then the following statements hold.
\begin{itemize}
\item[$(1)$]
	If $y$ is a {real zero}, then there exists 
	$g\in\SR$ such that\\ $f(x)=(x-y)\,g(x)$.
\item[$(2)$]
	If $y\in \OO_D\setminus\R$, then there exists 
	$h\in\SR$ and $a,b\in A$ such that $\langle a,b\rangle\subseteq N_A$ and $f(x)=\Delta_{y}(x)\,h(x)+xa+b$. Moreover,
\begin{itemize}
\item[$\bullet$]
$y$ is a {spherical zero} of $f$ if and only if $a=b=0$.
\item[$\bullet$]
$y$ is an $\SS_A$--isolated non--real zero of $f$ if and only if $a\ne0$ (in this case $y=-ba^{-1}$).
\end{itemize}
\end{itemize}
If there exists a real subspace $V\subseteq N_A$ such that $F(z)\in V\otimes\C$ for all $z\in D$, then $g$ and $h$ are admissible.
If $f$ is real, then $g,h$  are real and $a=b=0$.
\end{theorem}
\begin{proof}
We can suppose that $\text{conj}(D)=D$. In the proof, we will use the following fact. For every holomorphic function $F: D\rightarrow\Ac$, there is a unique decomposition $F=F^++F^-$, with $F^\pm$ holomorphic, $F^+$ complex intrinsic and $F^-$ satisfying $F^-(\bar z)=-\overline{F^-(z)}$. It is enough to set $F^{\pm}(z):=\frac12(F(z)\pm\overline{F(\bar z)})$.

Assume that $y$ is a real zero of $f=\I(F)$. Then $F(y)=0$ and therefore, by passing to a basis representation, we get $F(z)=(z-y)\,G(z)$ for a holomorphic mapping $G:D\rightarrow\Ac$. $G$ is complex intrinsic: $F(\bar z)=(\bar z-y)\,G(\bar z)=\overline{F(z)}=(\bar z-y)\,\overline{G(z)}$, from which it follows that $\overline{G(z)}=G(\bar z)$ for every $z\ne y$ and then on $D$ by continuity. Let $g=\I(G)\in\SR$. Then $f(x)=(x-y)\,g(x)$.
Assume that $F(z)\in V\otimes\C\,\forall z\in D$.
Since $x-y$ is real, $G_1(z)$ and $G_2(z)$ belong to $V\subseteq N_A$ for every $z\ne y$ and then also for $z=y$. Therefore $g$ is admissible.

Now assume that $y=\alpha+\beta J\in V(f)\setminus\R$. Let $\zeta=\alpha+i\beta\in D$. If $F(\zeta)\ne0$, there exists a unique $K\in\SS_A$ such that $F_1(\zeta)+KF_2(\zeta)=0$ (cf.\ Theorem~\ref{structure_thm}). Otherwise, let $K$ be any square root of $-1$. Let $F_K:=(1-iK)F$. Then $F_K$ is a holomorphic mapping from $D$ to $\Ac$, vanishing at $\zeta$:
\[F_K(\zeta)=F_1(\zeta)+iF_2(\zeta)-iKF_1(\zeta)+KF_2(\zeta)=0.\]
Then there exists a holomorphic mapping $G:D\rightarrow\Ac$ such that \[F_K(z)=(z-\zeta)\,G(z).\] Let $G_1(z)$ be the holomorphic mapping such that $G(z)-G(\bar\zeta)=(z-\bar\zeta)\,G_1(z)$ on $D$. Then
\[F_K(z)=(z-\zeta)((z-\bar\zeta)\,G_1(z)+G(\bar\zeta))=\Delta_y(z)\,G_1(z)+(z-\zeta)G(\bar\zeta).\]
Here $\Delta_y(z)$ denotes the complex intrinsic polynomial $z^2-z\,t(y)+n(y)$, which induces the characteristic polynomial $\Delta_y(x)$.

Let $F^+_K$ be the complex intrinsic part of $F_K$  
It holds $F^+_K=F$:
\[F^+_K(z)=\frac12(F_K(z)+\overline{F_K(\overline z)})=\frac{1-iK}2F(z)+\frac{1+iK}2\overline{F(\overline z)}=F(z).\]
Therefore
\[F(z)=\Delta_y(z)\,G_1^+(z)+\frac12((z-\zeta)\, G(\overline\zeta)+(z-\overline\zeta)\,\overline{G(\overline\zeta)})=\Delta_y(z)\,H(z)+za+b,\]
where $H=G_1^+$ is complex intrinsic and $a,b\in A$. Since  $F(\zeta)=\zeta a+b$, the elements $a=\beta^{-1}F_2(\zeta)$ and $b=F_1(\zeta)-\alpha\beta^{-1}F_2(\zeta)$ belong to $\langle F_1(\zeta),F_2(\zeta)\rangle\subseteq N_A$. 
If there exists $V\subseteq N_A$ such that $F(z)\in V\otimes\C\,\forall z\in D$, then it follows that $F(z)-za-b=\Delta_y(z)\,H(z)$ belongs to $V\otimes\C$ for all $z\in D$. From this and from the reality of $\Delta_y$, it follows, by a continuity argument, that $H(z)\in V\otimes\C\ \forall z\in D$, i.e.\ $h$ is admissible.

If $f$ is real, then $\langle F_1(z),F_2(z)\rangle\subseteq\R$ for every $z\in D$ and we get the last assertion of the theorem.
\end{proof}

\begin{remark}
A simple computation shows that 
\[\sd \Delta_y(x)=t(x)-t(y)\text{\quad and\quad}v_s \Delta_y(x)=\frac12t(x)(t(x)-t(y))-n(x)+n(y).\]
It follows that, for every non--real $y\in V(f)$, the element $a\in N_A$ which appears in the statement of the preceding theorem is the spherical derivative of $f$ at $x\in\SS_y$. 
\end{remark}

\begin{example}
The function $f(x)=1-IJ$ ($J$ fixed in $\SS_\HH$, $x=\alpha+\beta I\in\HH\setminus \R, \beta>0$) of Remark~\ref{rem8} of Section~\ref{sec:ZerosOfSliceFunctions} vanishes at $y=-J$. The division procedure gives
\[f(x)=(x^2+1)\,h(x)-xJ+1\]
with $h=\I(H)$ induced on $\HH\setminus\R$ by the holomorphic function
\[H(z)=\begin{cases}\frac J{z+i}\quad\text{on $\C^+=\{z\in\C\ |\ Im(z)>0\}$}\\\frac J{z-i}\quad\text{on $\C^-=\{z\in\C\ |\ Im(z)<0\}$}\end{cases}.\]
\end{example} 

\begin{remark}
In  \cite{GhPePreprint2009}, it was proved that when $A=\HH$ or $\Oc$, part (1) of the preceding theorem holds for {every} $y\in V(f)$:  $f(x)=(x-y)\cdot g(x)$. This can be seen in the following way (we refer to the notation used in the proof of the Remainder Theorem). From
\[F(z)=(z-\zeta)(z-\bar\zeta)\,H(z)+za+b=(z-y)((z-y^c)\,H(z)+a)+ya+b,\]
we get $f(x)=(x-y)\cdot g(x)+ya+b$, where $g=\I((z-y^c)\,H(z)+a)$. But $ya+b=0$ and then we get the result.
\end{remark}

\begin{corollary}\label{cor_remainder}
Let $f\in\SR$ be admissible.
If\, $\SS_y$ contains at least one zero of $f$, of whatever type, then $\Delta_y$ divides $N(f)$.
\end{corollary}
\begin{proof}
If $y\in\R$, then $\Delta_y(x)=(x-y)^2$. From the Theorem, $f=(x-y)\,g$, $f^c=(x-y)\,g^c$, and then
$N(f)=(x-y)^2\, g\cdot g^c=\Delta_y\, N(g)$.
If $y$ is a spherical zero, then $f=\Delta_y\,h$ and therefore $f^c=\Delta_y\,h^c$, $N(f)=f\cdot f^c=\Delta_y^2\, N(h)$.
If $y_1\in V(f)\cap\SS_y$ is an $\SS_A$--isolated, non--real zero, then 
\[f=\Delta_y\,h+xa+b,\quad f^c=\Delta_y\,h^c+xa^c+b^c,\]
from which it follows that 
\begin{align*}
N(f)&=\Delta_y\big[\Delta_y\,N(h)+h\cdot(xa^c+b^c)+(xa+b)\cdot h^c\big]+\\
&+x^2\,n(a)+x\,t(ab^c)+n(b).
\end{align*}
Since $y_1a+b=0$ with $\langle a,b\rangle$ contained in $N_A$, it follows that $n(b)=n(y_1)n(a)$, the trace of $ab^c$ is real, and
\[t(y_1)n(a)=(y_1+y_1^c)(aa^c)=(y_1a)a^c+a(a^cy_1^c)=-ba^c-ab^c=-t(ab^c).\]
Then $x^2\,n(a)+x\,t(ab^c)+n(b)=\Delta_{y_1} n(a)=\Delta_{y} n(a)$ and $\Delta_{y}\mid N(f)$.
\end{proof}

Let $f\in\SR$ be admissible. Let $y=\alpha+\beta J$, $\zeta=\alpha+i\beta$. Since $f$ is admissible, $N(f)$ is real and therefore if $N(f)=\Delta_y^s\, g$, also the quotient $g$ is real. The relation $N(f)=\Delta_y^s\, g$ is equivalent to the complex equality $CN(F)(z)=\Delta_\zeta(z)^s\,G(z)$, where $\Delta_\zeta(z)=z^2-z\,t(y)+n(y)=(z-\zeta)(z-\bar\zeta)$ and $\I(G)=g$. If $N(f)\not\equiv0$, then $CN(F)\not\equiv0$ and therefore $\zeta$ has a multiplicity as a (isolated) zero of the holomorphic function $CN(F)$ on $D$.

In this way, we can introduce, for any admissible slice regular function $f$ with $N(f)\not\equiv0$, the concept of multiplicity of its zeros.

\begin{definition}\label{def_mult}
Let $f\in\SR$ be admissible, with $N(f)\not\equiv0$.
Given a non--negative integer $s$ and an element $y$ of $V(f)$, we say that $y$ is a \emph{zero of $f$ of multiplicity $s$} if $\Delta_y^s \mid N(f)$ and $\Delta_y^{s+1} \nmid N(f)$. We will denote the integer $s$, called \emph{multiplicity of $y$}, by $m_f(y)$.
\end{definition}


In the case of $y$  real, the preceding condition is equivalent to $(x-y)^s\mid f$ and $(x-y)^{s+1}\nmid f$. If $y$ is a spherical zero, then $\Delta_y$ divides $f$ and $f^c$. Therefore $m_f(y)$ is at least 2.
If $m_f(y)=1$, $y$ is  called a \emph{simple zero of $f$}. 

\begin{remark}
In the case of  on quaternionic polynomials, the preceding definition is equivalent to the one given in \cite{BrayWhaples1983} and in \cite{GeSt2008Milan}. 
 \end{remark}

\subsection{Zeros of products}
\label{sec:ZerosOfProducts}

\begin{proposition}
Let $A$ be associative. Let $f,g\in\Sl(\OO_D)$. Then $V(f)\subseteq V(f\cdot g)$.
\end{proposition}
\begin{proof}
Assume that $f(y)=0$. Let $y=\alpha+\beta J$, $z=\alpha+i\beta$ (with $\beta=0$ if $y$ is real). Then $0=f(y)=F_1(z)+JF_2(z)$. Therefore
\begin{align*}
(f\cdot g)(y)&=(F_1(z)G_1(z)-F_2(z)G_2(z))+J(F_1(z)G_2(z)+F_2(z)G_1(z))=\\
&=(-JF_2(z)G_1(z)-F_2(z)G_2(z))+J(-JF_2(z)G_2(z)+F_2(z)G_1(z))=0.
\end{align*}
\end{proof}

As shown in \cite{Serodio2007} for octonionic polynomials and in \cite{GhPePreprint2009} for octonionic power series, if $A$ is not associative the statement of the proposition is no more true. However, we can still say something about the location of the zeros of $f\cdot g$. For quaternionic and octonionic power series, we refer to \cite{GhPePreprint2009} for the precise relation linking the zeros of $f$ and $g$ to those of $f\cdot g$. For the associative case, see also \cite[\S16]{Lam}.

\begin{proposition}\label{pro15}
Let $A$ be associative or $A=\Oc$.  Assume that $n(x)=n(x^c)\ne0$ for every $x\in A\setminus\{0\}$ such that $n(x)$ is real. If $f$ and $g$ are admissible slice functions on $\OO_D$, then it holds:
\[\bigcup_{x\in V(f \cdot g)}\is_{x} \; \; = \, \, \bigcup_{x\in V(f)\cup V(g)}\is_{x}\]
or, equivalently, given any $x \in \OO_D$, $V(f\cdot g)\cap \is_{x}$ is non--empty if and only if $\left(V(f)\cup V(g)\right)\cap\is_{x}$ is non--empty. In particular, the zero set of $f\cdot g$ is contained in the union  $\bigcup_{x\in V(f)\cup V(g)}\is_{x}$.
\end{proposition}
\begin{proof}
By combining Corollary~\ref{cor13}, Theorem~\ref{thmNfg}  and Corollary~\ref{cor15}, we get:
\begin{eqnarray*}
\textstyle \bigcup_{x\in V(f \cdot g)}\is_{x} &=& V(N(f \cdot g)) = V(N(f))\cup V(N(g))=
\bigcup_{x\in V(f)\cup V(g)}\is_{x}.
\end{eqnarray*}
Since $V(f \cdot g) \subseteq V(N(f \cdot g))$, the proof is complete.
\end{proof}

\begin{example}
Let $A=\R_3$.  Consider the polynomials  $f(x)=xe_2+e_1$, $g(x)=xe_3+e_2$ of Example~\ref{ex4} of Section~\ref{sec:NormalFunctionAndAdmissibility}. Then  $(f\cdot g)(x)=x^2e_{23}+x(e_{13}-1)+e_{12}$. The zero sets of $f$ and $g$ are $V(f)=\{e_{12}\}$, $V(g)=\{e_{23}\}$. The remainder theorem gives $f\cdot g=\Delta_{e_{12}}e_{23}+x(e_{13}-1)+(e_{12}-e_{23})$, from which we get the unique isolated zero $e_{12}$ of $f\cdot g$, of multiplicity 2.
\end{example}

\subsection{The Fundamental Theorem of Algebra}
\label{sec:FTA}
In this section, we focus our attention on the zero set of slice regular admissible \emph{polynomials}.
If $p(x)=\sum_{j=0}^m x^j a_j$ is an admissible polynomial of degree $m$ with coefficients $a_j\in A$, the normal polynomial 
\[\textstyle
N(p)(x)=(p \ast p^c)(x)=\sum_nx^n\big(\sum_{j+k=n}a_ja_k^c\big)=:\sum_n x^n c_n
\] 
has degree $2m$ and real coefficients. The last property is immediate when $Q_A=A$ (i.e.\ $A=\HH$ or $\Oc$), since
\[c_n=\sum_{j+k=n}a_ja_k^c=\sum_{j=0}^{\left[\frac{n-1}2\right]}t(a_ja_{n-j}^c)+n(a_{n/2})\]
(if $n$ is odd the last term is missing). In the general case, if $p$ is admissible and $p=\I(P)$, $N(p)$ is real, i.e.\ the polynomial $CN(P)(z)=\sum_{n=0}^{2m} z^n c_n$
is complex--valued. From this, it follows easily that the coefficients $c_n$ must be real. Indeed, $c_0=CN(P)(0)\in \C\cap A=\R$. Moreover, since $CN(P)(z)-c_0$ is complex for every $z$, also $c_1+zc_2+\cdots+z^{2m-1}c_{2m}\in\C$ and then $c_1\in \R$. By repeating this argument, we get that every $c_n$ is real.

\begin{theorem}[Fundamental Theorem of Algebra with multiplicities]
\label{FTA}
Let $p(x)=\sum_{j=0}^m x^j a_j$ be a polynomial of degree $m>0$ with coefficients in $A$. 
Assume that $p$ is admissible $($for example, this is true in the case in which
the real vector subspace $\langle a_0,\ldots,a_m\rangle$ of $A$ is contained in $N_A$$)$.
Then $V(p)=\{y\in Q_A\ |\ p(y)=0\}$ is non--empty.
More precisely, there are distinct ``spheres''\, $\SS_{x_1},\ldots, \SS_{x_t}$ such that 
\[V(p)\subseteq\bigcup_{k=1}^t \SS_{x_k}=V(N(p)),\quad V(p)\cap\SS_{x_j}\ne\emptyset \text{\quad for every $j$},\] 
and, for any choice of zeros\, $y_1\in\SS_{x_1}, \ldots,y_t\in\SS_{x_t}$ of $p$,
the following equality holds:
\[\sum_{k=1}^t m_p(y_k)=m.\]
\end{theorem}
\begin{proof}
Let $p=\I(P)$. 
As seen before, $N(p)$ is a polynomial with real coefficients and degree $2m$.
Let $J\in\SS_A$. Then the set $V(N(p)_J)=\{z\in\C_J\ |\ N(p)(z)=0\}=\C_J \cap \bigcup_{y \in V(p)}\SS_y$ is non--empty and contains at most $2m$ elements. Corollary~\ref{cor13} tells that $V(p)\cap\SS_y\ne\emptyset$ for every $y$ such that $V_J(N(p))\cap\SS_y\ne\emptyset$. Therefore, $V(p)$ is non--empty.

Let $y\in V(p)$. If there exists $g\in\Sl\Reg(Q_A)$ such that $N(p)=\Delta_y^s\, g$ and $\Delta_y\nmid g$, the slice function $g$ must necessarily be real, since $N(p)$ and $\Delta_y$ are real.   Then 
$N(N(p))=N(p)^2=\Delta_y^{2s}\, g^2$. Therefore $m_{N(p)}(y)\ge 2m_p(y)$. We claim that $m_{N(p)}(y)=2m_p(y)$. If $g^2=\Delta_y\, h$, with $h\in\Sl\Reg(Q_A)$ real, then $G(z)^2=\Delta_y(z)\, H(z)$ on $D$, with $\I(G)=g$ and $\I(H)=h$. Let $y=\alpha+\beta J$, $\zeta=\alpha+i\beta$. Then $G(\zeta)^2=\Delta_y(\zeta)\,H(\zeta)=0$ and therefore $G(\zeta)=0$. This implies that $g$ vanishes on $\SS_y$. From the Remainder Theorem, we get that $\Delta_y\mid g$, a contradiction. Then $\Delta_y\nmid g^2$ and $m_{N(p)}(y)=2m_p(y)$.

Every real zero of $N(p)_J$ corresponds to a real zero of $N(p)$ with the same multiplicity (cf.\ the remark made after Definition~\ref{def_mult}). The other zeros of $N(p)_J$ appears in conjugate pairs and correspond to spherical zeros of $N(p)$. Let $y=\alpha+\beta J$, $y^c=\alpha-\beta J$. If $N(p)_J(y)=N(p)_J(y^c)=0$, then $N(p)\equiv0$ on $\SS_y$. Moreover, since $\Delta_y(x)=(x-y)\cdot(x-y^c)$, the multiplicity $m_{N(p)}(y)$ is the sum of the (equal) multiplicities of $y$ and $y^c$ as zeros of $N(p)_J$. Therefore $2m=\sum_{k=1}^t m_{N(p)}(y_k)=2\sum_{k=1}^t m_p(y_k)$.
\end{proof}

\begin{remark}
Since the multiplicity of a spherical zero is at least $2$, if $r$ denotes the number of real zeros of an admissible polynomial $p$, $i$ the number of $\SS_A$--isolated non--real zeros of $p$ and $s$ the number of ``spheres'' $\SS_y$ ($y\notin\R$) containing spherical zeros of $p$, we have that $r+i+2s \leq \deg(p)$.
\end{remark}

\begin{examples}
$(1)$
Every polynomial $\sum_{j=0}^mx^ja_j$, with paravector coefficients $a_j$ in the Clifford algebra $\R_n$, has $m$ roots counted with their multiplicities in the quadratic cone $Q_A$. If the coefficients are real, then the polynomial has at least one root in the paravector space $\R^{n+1}$, since every ``sphere'' $\SS_y$ intersect $\R^{n+1}$ (cf.\ \cite[Theorem 3.1]{YangQianActa}).

$(2)$ 
In $\R_3$, the polynomial $p(x)=xe_{23}+e_1$ vanishes only at $y=e_{123}\notin Q_A$. Note that $p$ is \emph{not} admissible: $e_1,e_{23}\in N_A$, but $e_1+e_{23}\notin N_A$.

$(3)$
An admissible polynomial of degree $m$, even in the case of non spherical zeros, can have more than $m$ roots in the whole algebra. For example, $p(x)=x^2-1$ has \emph{four} roots in $\R_3$, two in the quadratic cone ($x=\pm1$) and two outside it ($x=\pm e_{123}$).

$(4)$ 
In $\R_3$, the admissible polynomial $p(x)=x^2+xe_3+e_2$ has two isolated zeros 
\[y_1=\frac12(1- e_2-e_3+e_{23}),\ y_2=\frac12(-1+e_2-e_3+e_{23})\]
in $Q_A\setminus\R^4$. They can be computed by solving the complex equation $CN(P)=z^4+z^2+1=0$ ($\I(P)=p$) to find the two ``spheres'' $\SS_{y_1}$, $\SS_{y_2}$ and then using the Remainder Theorem (Theorem~\ref{Remainder_Theorem}) with $\Delta_{y_1}=x^2-x+1$ and $\Delta_{y_2}=x^2+x+1$
(cf.\ \cite[Example 3]{YangQianActa}).

$(5)$ 
The reality of $N(p)$ is not sufficient to get the admissibility of $p$ (cf.\ Proposition~\ref{pro12}). For example, the polynomial
$p(x)=x^2e_{123}+x(e_1+e_{23})+1$ has real normal function $N(p)=(x^2+1)^2$, but the spherical derivative $\sd p=t(x)e_{123}+e_1+e_{23}$ has $N(\sd p)
$ not real. In particular, $\sd p(J)=e_1+e_{23}\notin N_A$ for every $J\in\SS_A$. The non--admissibility of $p$ is reflected by the existence of two distinct zeros on $\SS_A$, where $p$ does not vanish identically: $p(e_1)=p(e_{23})=0$. 
\end{examples}

\section{Cauchy integral formula for $C^1$ slice functions}
\label{sec:CauchyIntegralFormulaForSliceFunctions}
Let $\Sl^1(\overline \OO_D):=\{f=\I(F)\in\Sl(\OO_D)\ |\ F\in C^1(\overline D)\}$, where $\overline D$ denotes the topological closure of $D$.
For a fixed element $y=\alpha'+\beta' J\in Q_A$, let $\zeta=\alpha'+i\beta'\in\C$. The characteristic polynomial $\Delta_y$ is a real slice regular function on $Q_A$, with zero set $\SS_y$. For every $x=\alpha+\beta I\in Q_A\setminus \SS_y$, $z=\alpha+i\beta$, define
\[C_A(x,y):=\I\left(-\Delta_y(z)^{-1}(z-y^c)\right).\]
$C_A(x,y)$ is slice regular on $Q_A\setminus \SS_y$ and has the following property:
\[C_A(x,y)\cdot (y-x)=\I\left(-\Delta_y(z)^{-1}(z-y^c)(y-z)\right)=\I(1)=1.\]
This means that $C_A(x,y)$ is a slice regular inverse of $y-x$ w.r.t.\ the product introduced in Section~\ref{sec:Productofslicefunctions}. Moreover, when $x\in\C_J$ the product coincides with the pointwise one, and therefore
\[C_A(x,y)=(y-x)^{-1}\text{\quad for $x,y\in\C_J, x\ne y,y^c$}.\]
For this reason, we will call $C_A(x,y)$ the \emph{Cauchy kernel for slice regular functions on $A$}.
This kernel was introduced in \cite{CoGeSaPreprint2008} for the quaternionic case and in \cite{CoSaPreprint2008} for slice monogenic functions on a Clifford algebra. In \cite{CoSaStPreprint2009Mich}, the kernel was applied to get Cauchy formulas for $C^1$ functions (also called \emph{Pompeiu formulas}) on a class of domains intersecting the real axis. Here we generalize these results to $A$--valued slice functions of class $\Sl^1(\overline\OO_D)$.

\begin{theorem}[Cauchy integral formula]
Let $f\in\Sl^1(\overline \OO_D)$. Let $J\in\SS_A$ be any square root of $-1$. Assume that 
 $D_J=\OO_D\cap \C_J$ is a relatively compact open subset of $\C_J$, with boundary $\partial D_J$ piecewise of class $C^1$.
Then, for every $x\in \OO_D$ (in the case in which $A$ is associative) or every $x\in D_J$ (in the case in which $A$ is not associative), the following formula holds:
\[f(x)=\frac1{2\pi}\int_{\partial D_J}C_A(x,y)J^{-1}{dy}\,f(y)-\frac1{2\pi}\int_{D_J}C_A(x,y)J^{-1}{{dy^c}\wedge dy}\,\dd{f}{y^c}(y).\]
\end{theorem}
\begin{proof}
Let $x=\alpha+\beta J$ belong to $D_J$. Let $z=\alpha+i\beta$ and $f=\I(F)$. Let $F=F_1+iF_2=\sum_{k=1}^d F^ku_k$, with $F^k\in C^1(\overline D,\C)$, be the  representation of the stem function $F$ w.r.t.\ a basis $\B=\{u_k\}_{k=1,\ldots, d}$ of $A$. Denote by $\phi_J:\C\rightarrow \C_J$  the isomorphism sending $i$ to $J$. Define $F^k_J:=\phi_J\circ F^k\circ \phi_J^{-1}\in C^1(\overline{D}_J,\C_J)$. If $F_1^k,F_2^k$ are the real components of $F^k$, then $F_1=\sum_k F^k_1u_k$, $F_2=\sum_k F^k_2u_k$. Moreover, $F^k_J(x)=(\phi_J\circ F^k)(z)=F^k_1(z)+JF^k_2(z)$, from which it follows that $f(x)=F_1(z)+JF_2(z)=\sum_k F_J^k(x)u_k$.

The (classical) Cauchy formula applied to the $C^1$ functions $F^k_J$ ($k=1,\ldots,d$) on the domain $D_J$ in  $\C_J$,  gives:
\[F^k_J(x)=\frac1{2\pi J}\int_{\partial D_J}(y-x)^{-1}F^k_J(y)\,dy-\frac1{2\pi J}\int_{D_J}(y-x)^{-1}\dd{F^k_J(y)}{\bar y}\,d\bar y\wedge dy.\]
On $\C_J$ the conjugate $\bar y=\alpha-\beta J$ coincides with $y^c$ and the slice regular Cauchy kernel $C_A(x,y)$ is equal to the classical Cauchy kernel $(y-x)^{-1}$. Then we can rewrite the preceding formula as
\[F^k_J(x)=\frac1{2\pi}\int_{\partial D_J}C_A(x,y)\,J^{-1}dy\,F^k_J(y)-\frac1{2\pi}\int_{D_J}C_A(x,y)\, J^{-1}dy^c\wedge dy\,\dd{F^k_J(y)}{y^c}.\]

We use Artin's Theorem for alternative algebras.
For any $k$, the coefficients of the forms in the integrals above and the element $u_k$ belong to the subalgebra generated by $J$ and $u_k$. Using this fact and the equality \[\sum_k\dd{F^k_J}{y^c}u_k=\sum_k\dd{F^k_J}{\bar y}u_k=\I\left(\dd{F}{\bar \zeta}\right)=\dd {f}{y^c},\]
 after summation over $k$ we can write
\begin{align}\label{CF}
f(x)&=\sum_k F_J^k(x)u_k=
\\
&=\frac1{2\pi}\int_{\partial D_J}C_A(x,y)J^{-1}{dy}\,f(y)-\frac1{2\pi}\int_{D_J}C_A(x,y)J^{-1}{{dy^c}\wedge dy}\,\dd{f}{y^c}(y).\notag
\end{align}

Now assume that $A$ is associative and that $x=\alpha+\beta I\in\OO_D\setminus\C_J$. Let $x'=\alpha+\beta J$, $x''=\alpha-\beta J\in D_J$. Since $f$ and $C_A(\,\cdot\, ,y)$ are slice functions on $\OO_D$, their values at $x$ can be expressed by means of the values at $x'$ and $x''$  (cf.\ the representation formula \eqref{rep2}):
\[f(x)=\frac12\left(f(x')+f(x'')\right)-\frac{I}2\left(J\left(f(x')-f(x'')\right)\right)\]
and similarly for $C_A(x,y)$. By applying the Cauchy  formula $\eqref{CF}$ for $f(x')$ and $f(x'')$ and the representation formulas, we get that formula $\eqref{CF}$ is valid also at $x$. The associativity of $A$ allows to pass from the representation formula for $f(x)$ to the same formula for $C_A(x,y)$ inside the integrals.
\end{proof}

\begin{corollary}[Cauchy representation formula]
If $f\in\Sl^1(\overline \OO_D)$ is slice regular on $\OO_D$, then, for every $x\in \OO_D$ (in the case where $A$ is associative) or every $x\in D_J$ (in the case where $A$ is not associative),  
\[f(x)=\frac1{2\pi}\int_{\partial D_J}C_A(x,y)J^{-1}{dy}\, f(y).\]
If $f$ is real, the formula is valid for every $x\in\OO_D$ also when $A$ is not associative.
\end{corollary}

\begin{proof}
The first statement is immediate from the Cauchy integral formula. If $f$ is real and $A$ is not associative, for every $x\in\C_I\cap\OO_D$ we have, using the same notation of the proof of the theorem,
\begin{align*}
f(x)&=\frac12\left(f(x')+f(x'')\right)-\frac{I}2\left(J\left(f(x')-f(x'')\right)\right)=\\
&=\frac1{2\pi}\int_{\partial D_J}\frac{C_A(x',y)+C_A(x'',y)}2J^{-1}{dy}\,f(y)-\\
&\quad-I\left(\frac{J}{2\pi}\int_{\partial D_J}\frac{C_A(x',y)-C_A(x'',y)}2J^{-1}{dy}\,f(y)\right)=\\
&=\frac1{2\pi}\int_{\partial D_J}\left(\frac{C_A(x',y)+C_A(x'',y)}2-I\left(J\frac{C_A(x',y)-C_A(x'',y)}2\right)\right)J^{-1}{dy}\,f(y).
\end{align*}
The last equality is a consequence of the reality of $f$. If $y\in\C_J$, then $f(y)\in\C_J$ and $I(Jf(y))=(IJ)f(y)$ from Artin's Theorem. The representation formula applied to $C_A(x,y)$ gives the result.
\end{proof}

A version of the Cauchy representation formula for quaternionic power series was proved in \cite{GeSt2007Adv} and extended in \cite{CoGeSaPreprint2008}. For slice monogenic functions on a Clifford algebra, it was given in \cite{CoSaPreprint2008}.




\end{document}